\documentclass[a4paper]{amsart}

\usepackage[utf8]{inputenc}
\usepackage[T1]{fontenc}

\usepackage{amsthm, amssymb, amsmath, amsfonts, mathrsfs}

\usepackage{mathtools}
\usepackage[ocgcolorlinks, colorlinks=true, pdfstartview=FitV, linkcolor=blue, citecolor=blue, urlcolor=blue, pagebackref=false]{hyperref}

\usepackage{microtype}

\newtheorem{thm}{Theorem}[section]

\newtheorem{lem}[thm]{Lemma}

\theoremstyle{remark}

\renewcommand{\le}{\leqslant}

\renewcommand{\ge}{\geqslant}
\renewcommand{\leq}{\leqslant}
\renewcommand{\geq}{\geqslant}
\renewcommand{\subset}{\subseteq}
\newcommand{\mcl}{\mathcal}
\newcommand{\E}{\mathbb{E}}
\newcommand{\B}{\mathbb{B}}

\newcommand{\N}{\mathbb{N}}
\renewcommand{\L}{\mathscr{L}}
\newcommand{\K}{\mathscr{K}}

\newcommand{\Ll}{\left}
\newcommand{\Rr}{\right}
\newcommand{\1}{\mathbf{1}}
\newcommand{\R}{\mathbb{R}}

\newcommand{\Z}{\mathbb{Z}}
\renewcommand{\P}{\mathbb{P}}
\newcommand{\cP}{\mathscr{P}}

\newcommand{\ov}{\overline}

\newcommand{\eps}{\varepsilon}
\def\d{{\mathrm{d}}}

\newcommand{\msf}{\mathsf}

\newcommand{\h}{\mathsf{h}}

\newcommand{\tr}{\mathrm{Tr}}
\newcommand{\la}{\langle}
\newcommand{\ra}{\rangle}
\newcommand{\dr}{\partial}

\renewcommand{\emptyset}{\varnothing}
\newcommand{\T}{\mathbb{T}}

\renewcommand{\a}{\mathbf{a}}
\newcommand{\Id}{\mathbf{I}_d}
\renewcommand{\b}{\mathbf{b}}
\newcommand{\A}{\mathsf{A}}
\newcommand{\BB}{\mathsf{B}}
\newcommand{\ah}{{\overbracket[1pt][-1pt]{\a}}}

\newcommand{\G}{\mathcal{G}}
\newcommand{\Qh}{\mathsf{Q}}
\newcommand{\LL}{\mathcal{L}}
\DeclareMathOperator{\supp}{Supp}

\numberwithin{equation}{section}

\title{On generalized Gaussian free fields and stochastic homogenization}
\author{Yu Gu, Jean-Christophe Mourrat}

\address[Yu Gu]{Department of Mathematics, Building 380, Stanford University, Stanford, CA, 94305, USA}

\address[Jean-Christophe Mourrat]{ENS Lyon, CNRS, 46 allée d'Italie, 69007 Lyon, France}

\begin{document}
\begin{abstract}

We study a generalization of the notion of Gaussian free field (GFF). Although the extension seems minor, we first show that a generalized GFF does not satisfy the spatial Markov property, unless it is a classical GFF. In stochastic homogenization, the scaling limit of the corrector is a possibly generalized GFF described in terms of an ``effective fluctuation tensor'' that we denote by $\Qh$. We prove an expansion of $\Qh$ in the regime of small ellipticity ratio. This expansion shows that the scaling limit of the corrector is not necessarily a classical GFF, and in particular does not necessarily satisfy the Markov property.

\bigskip


\medskip


\end{abstract}
\maketitle
%
%
%
%
%
%%%%%%%%%%%%%%%%%%%%%%%%%%%%%%%%%%%%%%%%%%%%%%%%%%%%%%%%%%%%%%
%%%%%%%%%%%%%%%%%%%%%%%%%%%%%%%%%%%%%%%%%%%%%%%%%%%%%%%%%%%%%%
%
%
%
% 
\section{Introduction}

This paper is motivated by recent results in the stochastic homogenization of operators of the form $-\nabla \cdot \a \nabla $, where $x \mapsto \a(x)$ is a random stationary field over $\R^d$ with strong decorrelation properties, and valued in the set $\mathsf{Sym}_d^+$ of $d$-by-$d$ symmetric positive-definite matrices. %In a discrete-space setting and under the condition that the coefficient field $\a$ is uniformly elliptic and has an underlying i.i.d.\ structure, the large-scale behavior of the fluctuations of the solutions of the associated Poisson equation was obtained in \cite{gu-mourrat}. A related result is the scaling limit of the corrector, obtained in \cite{MN}. 
In this context, we aim to obtain a precise description of the large-scale behavior of solutions of equations involving this operator, and in particular of the corrector. 

The corrector in the direction $\xi \in \R^d$ can be defined, up to an additive constant, as the unique sub-linear function $\phi_\xi : \R^d \to \R$ such that $-\nabla \cdot \a (\xi + \nabla \phi_\xi) = 0$. 
The minimizer in $H^1_0(U)$ of the mapping 
$v \mapsto \int_U (\xi + \nabla v) \cdot \a (\xi + \nabla v)$, where $U$ is a large domain, offers a finite-volume approximation of the corrector. This minimization of a random quadratic functional suggests a parallel with gradient Gibbs measures, where a determinimistic functional of quadratic type is used to produce a probability measure via the Gibbs principle. Such gradient Gibbs measures are known to rescale to Gaussian free fields (GFFs), see \cite{nadspe,gos,mil}. It is therefore natural to conjecture the same scaling limit for the corrector \cite{berbis}. 

Under assumptions described more precisely below, the scaling limit of the corrector was identified in \cite{mourrat2014correlation,MN} as an element of a slightly larger class of random fields, which we will call \emph{generalized} GFFs (see also \cite{gu-mourrat}).
%A version of this result was also obtained for solutions of the Poisson equation in \cite{gu-mourrat}. We refer to \cite{gu-mourrat} for more precision and for a heuristic derivation of these results.
Intuitively, we define the generalized GFF as the field $\Phi$ such that $\nabla \Phi$ is ``as close as possible'' to a white noise vector field. More precisely, a generalized GFF is a random field $\Phi$ solving 
\begin{equation}
\label{e.def.Phi}
-\nabla \cdot \ah \nabla \Phi = \nabla \cdot W,
\end{equation}
where $\ah \in \mathsf{Sym}_d^+$ and $W$ is a vector-valued white noise field with covariance matrix $\Qh \in \mathsf{Sym}_d^+$. For simplicity, we will only consider such equations in the full space $\R^d$ with $d \ge 3$, but our arguments also cover domains with suitable boundary condition, or the addition of a massive term, with only minor modifications. We recall that a classical GFF is a Gaussian field whose covariance function is the Green function of $-\nabla \cdot \mathbf{b} \nabla$, for some $\mathbf{b} \in \mathsf{Sym}_d^+$. When $\ah$ and $\Qh$ are proportional, the field $\Phi$ defined by \eqref{e.def.Phi} is a classical GFF, and thus satisfies the spatial Markov property \cite{sheffield}.

\medskip

The goal of this paper is twofold. In Section~\ref{s.markov}, we show that the Markov property does \emph{not} extend to the wider class of generalized GFFs. In fact, no generalized GFF satisfies the Markov property, unless it is a classical GFF. Therefore, our apparently mild extension of the notion of GFF in fact wipes out one of its fundamental properties. In Section~\ref{s.homog}, we turn to the context of homogenization, and study the effective parameters describing the generalized GFF arising as the scaling limit of the corrector. We prove an asymptotic expansion for these coefficients, in the regime of small ellipticity contrast. This expansion is of independent interest, and shows that the scaling limit of the corrector is generically \emph{not} a classical GFF, thereby justifying the relevance of the notion of generalized GFF.

\medskip

We close this introduction with two remarks. First, in order to substantiate that the definition \eqref{e.def.Phi} is consistent with the intuition that $\nabla \Phi$ is chosen to be ``as close as possible'' to a white noise vector field, it is simpler to discuss the case where the underlying state space is the torus $\T^d := [0,1)^d$. Endowing the space $L^2(\T^d,\R^d)$ with the scalar product
\begin{equation*}  %\label{e.}
(F, G) \mapsto \int_{\T^d} F \cdot \ah^{-1} G,
\end{equation*}
we denote by $L^2_{\mathrm{pot}}(\T^d)$ the closure in $L^2(\T^d,\R^d)$ of the set
\begin{equation*}  %\label{e.}
\{\ah \nabla f \ : \  f \in C^\infty(\T^d)\}
\end{equation*}
and denote by $L^2_{\mathrm{sol}}(\T^d)$ its orthogonal complement
\begin{equation*}  %\label{e.}
L^2_{\mathrm{sol}}(\T^d) := \{\mathbf{g} \in L^2(\T^d,\R^d) \ : \ \forall f \in C^\infty(\T^d), \ \int_{\T^d} \nabla f \cdot \mathbf{g} = 0\},
\end{equation*}
which is the set of solenoidal (i.e.\ divergence-free) vector fields. This provides us with the Helmholtz-Hodge decomposition of $L^2(\T^d,\R^d)$ into the orthogonal sum of $L^2_{\mathrm{pot}}(\T^d)$ and $L^2_{\mathrm{sol}}(\T^d)$. If $W$ were smooth, this would allow us to interpret $\ah \nabla \Phi$ as the orthogonal projection of $W$ onto $L^2_{\mathrm{pot}}(\T^d)$. For non-smooth $W$, the intepretation remains valid by testing and duality.

\medskip

Second, we point out that although the generalized GFF fails to satisfy the Markov property, it does satisfy a related domain decomposition property: there exists a family of random fields $(\Phi_A)$ on $\R^d$ indexed by all borel subsets $A$ of $\R^d$ such that the following holds for every Borel set $A \subset \R^d$:
\begin{itemize}  %\label{}
\item we have $\Phi = \Phi_A + \Phi_{A^c}$, ($A^c$ denotes the complement of $A$ in $\R^d$);
\item the random fields $\Phi_A$ and $\Phi_{A^c}$ are independent;
\item the random field $\Phi_A$ is an $\ah$-harmonic function in the interior of $A^c$. 
\end{itemize}
Indeed, these properties are easy to verify from the definition of $\Phi_A$ as solving
\begin{equation*}  %\label{e.}
-\nabla \cdot \ah \nabla \Phi_A = \nabla \cdot (W \1_{A}).
\end{equation*}

\section{Non-Markov property}
\label{s.markov}

In this section, we show that a generalized GFF satisfies the Markov property if and only if it is a classical GFF. 

\medskip

For $k,l$ positive integers, we write $C^\infty_c(\R^k,\R^l)$ for the space of infinitely differentiable functions from $\R^k$ to $\R^l$ with compact support. For every $U \subset \R^k$,
we let $C^\infty_c(U,\R^l) = \{f \in C^\infty_c(\R^k,\R^l) : \supp f \subset U\}$, where $\supp f$ denotes the support of $f$. We simply write $C^\infty_c(U) = C^\infty_c(U,\R)$ and $C^\infty_c = C^\infty_c(\R)$.  
 We fix $d \ge 3$ and $\ah, \Qh \in \mathsf{Sym}_d^+$. We say that the random distribution $W = (W_1,\ldots,W_d)$ is a white noise vector field with covariance matrix $\Qh$ if for every $\mathbf{f} = (\mathbf{f}_1,\ldots,\mathbf{f}_d) \in C^{\infty}_c(\R^d, \R^d)$, the random variable
$$
W(\mathbf{f}) := W_1(\mathbf{f}_1) + \cdots + W_d(\mathbf{f}_d)
$$
is a centered Gaussian with variance $\int_{\R^d} \mathbf{f} \cdot \Qh \mathbf{f}$. We denote by $(\Omega, \mathsf{F},\P)$ the underlying probability space,  and by $\E$ the associated expectation. Informally,
$$
\E[W_i(x) \, W_j(y)] = \Qh_{ij} \delta(x-y),
$$
where $\delta$ is the Dirac distribution. One can extend the set of admissible test functions for $W$ to every element of $L^2(\R^d,\R^d)$ by density. We define the solution $\Phi$ to \eqref{e.def.Phi} to be the random distribution such that for every $f \in C^\infty_c(\R^d)$,
\begin{equation}
\label{e.def.true.Phi}
\Phi(f) =  - W(\nabla (-\nabla \cdot \ah \nabla)^{-1} f),
\end{equation}
where $(-\nabla \cdot \a \nabla)^{-1} f$ is the unique function $u$ tending to zero at infinity and such that $-\nabla \cdot \a \nabla u = f$.
Formally integrating by parts shows the consistency between \eqref{e.def.Phi} and \eqref{e.def.true.Phi}. The latter makes sense since for $f \in C^\infty_c(\R^d)$, the function $\nabla (-\nabla \cdot \ah \nabla)^{-1} f$ is in $L^2(\R^d, \R^d)$, as can be checked for instance using the Green representation formula.

For every open $U \subset \R^d$, we define
$$
\mcl F(U) := \sigma\{\Phi(f), f \in C^\infty_c(U)\}.
$$
We simply write $\mcl F := \mcl F(\R^d)$. For every closed $A \subset \R^d$, we define
$$
\mcl F(A) := \bigcap_{\text{open } U \supseteq A} \mcl F(U).
$$
We enlarge the $\sigma$-algebras defined above (without changing the notation) so that they contain all $\P$-negligible measurable sets. We say that the field $\Phi$ is \emph{Markovian with respect to the open set $U$} if conditionally on $\mcl F(\partial U)$, the $\sigma$-algebras $\mcl F(\ov U)$ and $\mcl F(U^c)$ are independent (where $\ov U$ denotes the closure of $U$, $U^c$ the complement of $U$, and $\partial U = \ov U \cap U^c$).
\begin{thm}
\label{t.non-markov}
Let $U$ be an open subset of $\R^d$, $U \notin \{\emptyset, \R^d\}$. The field $\Phi$ is Markovian with respect to $U$ if and only if $\Phi$ is a classical GFF.
\end{thm}
Let $\LL:= (-\nabla \cdot \ah \nabla) (-\nabla \cdot \Qh \nabla)^{-1} (-\nabla \cdot \ah \nabla)$. It follows from \eqref{e.def.true.Phi} that for every $f,g \in C^\infty_c$,
\begin{equation}
\label{e.covariance}
\E[\Phi(f) \, \Phi(g)] = \int f \, \LL^{-1} \, g.
\end{equation}
In particular, if $\ah$ and $\Qh$ are proportional, then the field $\Phi$ is a classical GFF, and we recall that in this case the Markov property is well-known. In order to prove the theorem, it thus suffices to show that if $\Phi$ is Markovian with respect to $U$, then $\ah$ and $\Qh$ are proportional.

\medskip

To our knowledge, the rigorous study of the Markov property of random fields was initiated with \cite{mck}, where L\'evy's ``Brownian motion'' indexed by a multidimensional parameter \cite{levy} is shown to be Markovian if and only if the space dimension is odd. 

For random fields on discrete graphs, the Markov property is equivalent to the locality of the ``energy function'' (in the Gaussian case, this is the Dirichlet form, and more generally, we mean the logarithm of the probability density, up to a constant). This equivalence can be checked by a direct computation for Gaussian fields, and we refer to \cite[Theorems~4.1 and 4.2]{kolfried} for a more general statement. It is natural to expect a similar phenomenon in the continuum.  However, a counter-example to this conjecture was given in \cite{kotpre}. We will recall this counter-example in subsection~\ref{ss.counter}. In spite of this, relying on the fact that the field $\Phi$ is ``at least as regular as white noise'', we will be able to justify this conjecture in our context. 

Prior to the counter-example of \cite{kotpre}, incorrect proofs of the general conjecture were published. In \cite{pit1,kuensch}, the arguments are laid down pretending that the field is defined pointwise, and therefore the difficulty caused by the possibly low regularity of the field is missed. The paper \cite{kalman} proceeds more carefully, but is also flawed\footnote{there, the set $\mathfrak{M}(D_-)$ should be closed in order for \cite[Lemma~1]{kalman} to hold; but then the property $F \in \mathfrak{M}(D_-) \implies \supp F \subset D_-$ used in the proof of \cite[Lemma~2]{kalman} is false in general.}. 

The equivalence between locality of $\LL$ and the Markov property was also investigated in the framework of Dirichlet form theory \cite{dyn,roc,kol}. The arguments given there rely on potential theory, and only apply to operators satisfying the maximum principle. This is unfortunately not the case of the operator $\LL$ we consider here.  In order to see this, we can use the general result of \cite{courrege}, which identifies the class of integro-differential operators satisfying the maximum principle. In our case, the operator is moreover invariant under translation. In this setting, the results of \cite{courrege} can be understood as follows: if the operator satisfies the maximum principle, then it is the generator of a L\'evy process. Moreover, this L\'evy process must have the same scale invariance as Brownian motion, in view of the scaling properties of $\LL$. By the L\'evy-Khintchine formula, it follows that such a L\'evy process must be a multiple of Brownian motion. The operator $\LL$ must therefore be of the form $-\nabla \cdot \mathbf{b} \nabla$ for some matrix $\mathbf{b}$, but we show below that this can only happen when $\ah$ and $\Qh$ are proportional.

\medskip

\subsection{Proof of Theorem~\ref{t.non-markov}}
To sum up: using some mild regularity property of the field $\Phi$ defined by \eqref{e.def.true.Phi}, we show that the Markov property implies that $\LL$ is a local operator. We then observe that $\LL$ is not a local operator, unless $\ah$ and $\Qh$ are proportional. As announced in the theorem, we will be sufficiently careful to actually only require the Markov property with respect to one non-trivial open set. %The arguments we will present are certainly well-known to specialists, but as explained above, we could not find a suitable reference for them.

\medskip

We denote by $\mcl H$ the Hilbert space obtained as the closure of $\{\Phi(f), f \in C^\infty_c\}$ in $L^2(\Omega,\mcl F,\P)$. We denote by $\mcl K$ the Hilbert space obtained by completing $\{\LL^{-1} f, f \in C^\infty_c\}$ with respect to the scalar product 
$$
\la f,g \ra_\LL := \int f \LL g.
$$
The space $\mcl K$ is usually called the reproducing kernel Hilbert space. Note that $f \mapsto \la f, f \ra_\LL^{1/2}$ is within multiplicative constants of the homogeneous $H^1$ norm. 
By \eqref{e.covariance}, the spaces $\mcl H$ and $\mcl K$ are isometric:
\begin{lem}
\label{l.isometry}
The mapping
$$
\Ll\{
\begin{array}{rcl}
\{\Phi(f), f \in C^\infty_c\} & \to & \mcl K \\
\Phi(f) & \mapsto & \LL^{-1} f
\end{array}
\Rr.
$$
is well-defined and extends to an isometry from $\mcl H$ onto $\mcl K$.
\end{lem}
For every open $U \subset \R^d$, we let $\mcl H(U)$ be the closure of $\{\Phi(f), f \in C^\infty_c(U)\}$ in $\mcl H$, and $\mcl K(U)$ be the closure of $\{\LL^{-1} f, f \in C^\infty_c(U)\}$ in $\mcl K$. The isometry of Lemma~\ref{l.isometry} induces an isometry between $\mcl H(U)$ and $\mcl K(U)$. If $A \subset \R^d$ is a closed set, we define
$$
\mcl H(A) := \bigcap_{\text{open } U \supseteq A} \mcl H(U), \qquad \mcl K(A) := \bigcap_{\text{open } U \supseteq A} \mcl K(U).
$$

The following lemma is classical.

\begin{lem}
\label{l.proj}
Let $U$ be an open set. 

\noindent (1) The set $\mcl H(U)$ is the set of $\mcl F(U)$-measurable elements of $\mcl H$.

\noindent (2) The map
$$
\Ll\{
\begin{array}{rcl}
\mcl H & \to & \mcl H \\
X & \mapsto & \E[X \ | \ \mcl F(U)]
\end{array}
\Rr.
$$
is the orthogonal projection onto $\mcl H(U)$.
\end{lem}
\begin{proof}
Every element of $\mcl H(U)$ is $\mcl F(U)$-measurable. Conversely, if $X \in \mcl H(U) ^\perp$, then for every $f_1, \ldots, f_k \in C^\infty_c(U)$, we have $\E[X \, \Phi(f_i)] = 0$. Since $\mcl H$ is a Gaussian space, this implies that $X$ is independent of $(\Phi(f_1),\ldots,\Phi(f_k))$, and thus $X$ is independent of $\mcl F(U)$. We have shown
$$
\mcl H(U)^\perp \subset \{Y \in \mcl H \ : \ Y \mbox{ is $\mcl F(U)$-measurable}\}^\perp.
$$
Since both spaces are closed, this completes the proof of (1). Part (2) follows at once, since for every $X \in \mcl H$, we have $\Ll(X - \E[X \ | \ \mcl F(U)] \Rr) \in \mcl H(U)^\perp$.
\end{proof}

%We say that the field $\Phi$ is \emph{Markovian with respect to the open set $U$} if conditionally on $\mcl F(\partial U)$, the $\sigma$-algebras $\mcl F(\ov U)$ and $\mcl F(U^c)$ are independent (where $\ov U$ denotes the closure of $U$, $U^c$ the complementary of $U$, and $\partial U = \ov U \cap U^c$).

\begin{lem}
\label{l.decompH}
Let $\mcl H_0(\ov U)$ denote the orthogonal complement of $\mcl H(\dr U)$ in $\mcl H(\ov U)$. If $\Phi$ is Markovian with respect to the open set $U$, then 
$$
\mcl H = \mcl H_0(\ov U) \stackrel{\perp}{\oplus} \mcl H(\dr U) \stackrel{\perp}{\oplus} \mcl H_0(U^c).
$$
\end{lem}
\begin{proof}
We decompose the proof into three steps. 

\medskip

\noindent \emph{Step 1.}
We first show that 
\begin{equation}
\label{e.orthogonal-proj}
\mbox{ the orthogonal projection of $\mcl H(U^c)$ onto $\mcl H(\ov U)$ is $\mcl H(\dr U)$}.
\end{equation}

By Lemma~\ref{l.proj}, it suffices to show that for every $\mcl F(U^c)$-measurable $X \in \mcl H$,
$$
\E[X \, | \, \mcl F(\dr U)]  = \E[X \, | \, \mcl F(\ov U)].
$$
We show that
\begin{equation}
\label{e.measurability}
\E \Ll[ \Ll(\E[X \, | \, \mcl F(\dr U)] - \E[X \, | \, \mcl F(\ov U)]\Rr)^2 \Rr] = 0.
\end{equation}
By the Markov property and the inclusion $\mcl F(\dr U) \subset \mcl F(\ov U)$,
\begin{equation}
\label{e.comp1}
\E\Ll[ X \ \E[X \, | \, \mcl F(\ov U)] \, \big | \, \mcl F(\dr U)\Rr] = \Ll(\E[X \, | \, \mcl F(\dr U)]\Rr)^2.
\end{equation}
In particular,
\begin{eqnarray*}
\E \Ll[ \Ll(\E[X \, | \, \mcl F(\ov U)] \Rr)^2 \Rr] & = & \E \Ll[ X \ \E[X \, | \, \mcl F(\ov U)]  \Rr] \\
& \stackrel{\eqref{e.comp1}}{=}&  \E \Ll[ \Ll(\E[X \, | \, \mcl F(\dr U)] \Rr)^2 \Rr] %\\
%& = & \E \Ll[\E[X \, | \, \mcl F(\dr U)] \ \E[X \, | \, \mcl F(\ov U)]\Rr]
,
\end{eqnarray*}
and this proves \eqref{e.measurability}.

\medskip

\noindent \emph{Step 2.} 
We show that
$$
\mcl H_0(\ov U) \stackrel{\perp}{\oplus} \mcl H(\dr U) \stackrel{\perp}{\oplus} \mcl H_0(U^c) = \mcl H(\ov U) + \mcl H(U^c).
$$
The equality
$$
\mcl H_0(\ov U) + \mcl H(\dr U) + \mcl H_0(U^c) = \mcl H(\ov U) + \mcl H(U^c)
$$
is clear. The orthogonality of the sets on the left-hand side follows from the definition of $\mcl H_0$ and the previous step. In particular, the set $\mcl H(\ov U) + \mcl H(U^c)$ is closed.

\medskip

\noindent \emph{Step 3.} 
Let $f \in C^\infty_c$. In order to complete the proof, it suffices to see that 
\begin{equation}
\label{e.to-show-reg}
\Phi(f) \in \mcl H(\ov U) + \mcl H(U^c).
\end{equation} 
This is the step of the proof where we need to use the fact that $\Phi$ is ``at least as regular as white noise''. In more precise words, we show that we can control the quadratic form $g \mapsto \int g \LL^{-1} g$ over local functions $g$ by the $L^2$ norm of $g$ squared, and that this is sufficient to conclude. Without loss of generality, we assume that the support of $f$ is contained in the unit ball $B(0,1)$. Let $\chi : \R \to [0,1]$ be a $C^\infty$ function such that $\chi = 1$ on $(-\infty,0]$ and $\chi = 0$ on $[1,+\infty)$, and let
$$
\chi_n := 
\Ll\{
\begin{array}{rcl}
\R^d & \to & [0,1] \\
x & \mapsto & \chi(n \, \mathsf{dist}(x,U)),
\end{array}
\Rr.
$$
where $\mathsf{dist}(x,U)$ denotes the distance between $x$ and the open set $U$.
Note that
$$
\Phi(f) = \Phi(f \chi_n) + \Phi(f[1-\chi_n]),
$$
and that $\Phi(f[1-\chi_n]) \in \mcl H(U^c)$. We now argue that $\Phi(f \chi_n)$ converges to an element of $\mcl H(\ov U)$ as $n$ tends to infinity. This would imply that $\Phi(f[1-\chi_n])$ converges to an element of $\mcl H(U^c)$ and therefore complete the proof of \eqref{e.to-show-reg}. If the sequence $(\Phi(f\chi_n))$ converges in $\mcl H$, then the limit is necessarily in $\mcl H(\ov U)$; so what needs to be argued is simply the convergence of $\Phi(f[1-\chi_n])$ in $\mcl H$. %We note that $f \chi_n$ is a converging sequence in $L^2(\R^d)$. 
By \eqref{e.covariance}, for every $g \in C^\infty_c$,
$$
\E[\Phi(g)^2] = \int g \LL^{-1} g,
$$
so we need to check that $f \chi_n$ is a Cauchy sequence with respect to the seminorm
$$
g \mapsto \Ll(\int g \LL^{-1} g \Rr)^{\frac 1 2}.
$$
Letting $\Lambda$ be the largest eigenvalue of $\Qh$ and $h := (-\nabla \cdot \a \nabla)^{-1}g$, we have
$$
\int g \LL^{-1} g  \le \Lambda \int |\nabla h|^2 
%\le \Lambda^2 \int gh \le \Lambda^2 \Ll( \int g^2 \Rr)^{\frac 1 2} \Ll( \int h^2 \Rr)^{\frac 1 2}
.
$$
Since $d\geq 3$, using the Green function representation, we obtain that there exists a constant $C < \infty$ (depending only on $\a$) such that
$$
|\nabla h|(x)  \le C \int_{\R^d} \frac 1 {|x-y|^{d-1}} |g|(y) \, \d y.
$$
By Young's convolution inequality, the $L^2$ norm of the function
$$
x \mapsto \int_{\R^d} \frac {\1_{|x-y| \le 1}} {|x-y|^{d-1}} |g|(y) \, \d y
$$
is bounded by a constant times $\|g\|_{L^2}$. For functions $g$ with support in $B(0,1)$, we also have
$$
\int_{\R^d} \frac {\1_{|x-y| > 1}} {|x-y|^{d-1}} |g|(y) \, \d y \le C \frac{\|g\|_{L^1}}{1+|x|^{d-1}} \le C \frac{\|g\|_{L^2}}{1+|x|^{d-1}}.
$$
%and for such functions, we also have $\|g\|_{L^1} \le C \|g \|_{L^2}$. 

To sum up, we have shown that there exists a constant $C < \infty$ such that for every $g \in C^\infty_c$ with support in $B(0,1)$, 
$$
\Ll( \int g \LL^{-1} g \Rr)^{\frac 1 2} \le C \|g\|_{L^2}.
$$
Since $f\chi_n \to f$ in $L^2$, this completes the proof.
%
%Let $X \in \mcl H(U^c)$. We can decompose $X$ into $Y+Z$, with $Y \in \mcl H(\ov U)$ and $Z \in \mcl H(\ov U)^\perp$. By Step 1, we have $Y \in \mcl H(\dr U)$, so
%$$
%\mcl H(U^c) \subset \mcl H(\dr U)  \stackrel{\perp}{\oplus} \mcl H(\ov U)^\perp
%$$
%(the orthogonality being due to the inclusion $\mcl H(\dr U) \subset \mcl H(\ov U)$). By symmetry, we also have
%$$
%\mcl H(\ov U) \subset \mcl H(\dr U)  \stackrel{\perp}{\oplus} \mcl H(U^c)^\perp,
%$$
%which implies
%$$
%\mcl H(U^c) \cap \mcl H(\dr U)^\perp \subset \mcl H(\ov U)^\perp.
%$$
%Hence, $\mcl H(\ov U)^\perp = \mcl H_0(U^c)$, $\mcl H(U^c)^\perp = \mcl H_0(\ov U)$, and the lemma is proved.
%
%
%
%%
%Let $X \in \mcl H$. We define $Z \in \mcl H$ by
%$$
%X = \E[X \ | \ \mcl F(U)] + Z.
%$$
%By construction, $Z$ is orthogonal to any $\mcl F(U)$-measurable random variable. Since $\mcl H$ is a Gaussian space, this implies that $Z$ is independent of $\mcl F(U)$. Let $O$ be an open set containing $U^c$. For every $f \in C^\infty_c$, $\Phi(f)$ is $\sigma\Ll(\mcl F(U), \mcl F(O)\Rr)$-measurable. Hence, 
%$$
%\mcl F = \sigma\Ll(\mcl F(U), \mcl F(O)\Rr).
%$$
%\jccomment{this is wrong!!}
%Since $Z$ is independent of $\mcl F(U)$, we have
%$$
%Z = \E[ Z \ | \ \mcl F(O)].
%$$
%That is, $Z$ is $\mcl F(O)$-measurable, and thus in $\mcl H(O)$ by Lemma~\ref{l.proj}. As a consequence,
%$$
%Z \in \bigcap_{\text{open } O \supseteq U^c} \mcl H(O) = \mcl H(U^c).
%$$ 
%
\end{proof}

\begin{lem}
\label{l.local}
If $\Phi$ is Markovian with respect to the open set $U$, and if $f_1 \in C^\infty_c(U)$ and $f_2 \in C^\infty_c(\ov U^c)$, then $\la f_1,f_2 \ra_\LL = 0$.
\end{lem}
\begin{proof}
The isometry of Lemma~\ref{l.isometry} transports the decomposition of $\mcl H$ in Lemma~\ref{l.decompH} into a decomposition of $\mcl K$. In particular, 
$$
\mcl K(U^c)^\perp \subset \mcl K(\ov U).
$$
Let $f_1 \in C^\infty_c(U)$, and let $S := \supp f_1$. Note that
$$
\bigcap_{\substack{\text{open } O \supseteq U^c \\ O \cap S = \emptyset}} \mcl K(O) = \mcl K(U^c).
$$ 
Let $g \in C^\infty_c$ with support disjoint from $S$. We have
$$
0 = \int f_1 \, g = \la f_1, \LL^{-1} g \ra_\LL.
$$
By density, we deduce that $f_1 \in \mcl K(U^c)^\perp \subset \mcl K(\ov U)$. By symmetry, we also have $f_2 \in \mcl K(\ov U)^\perp$, which completes the proof.
\end{proof}

\begin{lem}
\label{l.non-loc}
Let $U$ be an open subset of $\R^d$, $U \notin \{\emptyset, \R^d\}$. If $\ah$ and $\Qh$ are not proportional, then there exist $f_1 \in C^\infty_c(U), f_2 \in C^\infty_c(\ov U^c)$ such that $\la f_1, f_2 \ra_\LL \neq 0$. 
\end{lem}
\begin{proof}

In the Fourier domain, $\la f_1,f_2\ra_\LL$ can be written explicitly as 
\[
\begin{aligned}
\la f_1,f_2\ra_\LL=&\int_{\R^d}f_1(x)(-\nabla \cdot \ah \nabla) (-\nabla \cdot \Qh \nabla)^{-1} (-\nabla \cdot \ah \nabla)f_2(x)dx\\
=&\frac{1}{(2\pi)^d}\int_{\R^d} \overline{\hat{f}_1(\xi)} \frac{(\xi\cdot  \ah\xi)^2}{\xi\cdot \Qh\xi}\hat{f}_2(\xi)d\xi.
\end{aligned}
\]
We change $\xi$ to $\Qh^{-\frac12} \xi$, and then go back to physical domain, to obtain that
\[
\begin{aligned}
&\la f_1,f_2\ra_\LL\\
=&|\Qh|^\frac12\int_{\R^d}(-\nabla \cdot  \Qh^{-\frac12}\ah\Qh^{-\frac12}\nabla) f_1(Q^\frac12 x)(-\Delta)^{-1}(-\nabla \cdot  \Qh^{-\frac12}\ah\Qh^{-\frac12}\nabla)f_2(Q^\frac12x)dx.
\end{aligned}
%\int_{\R^d} \overline{\hat{f}_1(\Qh^{-\frac12}\xi)} \frac{(\xi\cdot  \Qh^{-\frac12}\ah\Qh^{-\frac12}\xi)^2}{|\xi|^2}\hat{f}_2(\Qh^{-\frac12}\xi)d\xi
\] 
Since $\ah$ and $\Qh$ are not proportional, $
\Qh^{-\frac12}\ah\Qh^{-\frac12}$ is not a multiple of identity. By denoting $\A=\Qh^{-\frac12}\ah\Qh^{-\frac12}$, we only need to show that there exists $g_1,g_2$ whose supports are in $Q^\frac12 U$ and $Q^\frac12\ov U^c$ respectively and such that
\[
\int_{\R^d}(-\nabla \cdot \A\nabla) g_1(x)(-\Delta)^{-1}(-\nabla \cdot  \A\nabla)g_2(x)dx\neq 0.
\]

For $x$ away from the support of $g_2$, we can write
\[
\begin{aligned}
(-\Delta)^{-1}(-\nabla \cdot  \A\nabla)g_2(x)=& \int_{\R^d}\G(x-y)(-\nabla \cdot  \A\nabla g_2)(y)dy\\
=&\int_{\R^d} (-\nabla \cdot  \A\nabla \G)(x-y) g_2(y)dy
\end{aligned}
\]
with $\G$ the Green function of $-\Delta$. Since $g_1,g_2$ have disjoint supports, we can further integrate by parts to obtain
\[
\begin{aligned}
&\int_{\R^d}(-\nabla \cdot \A\nabla) g_1(x)(-\Delta)^{-1}(-\nabla \cdot  \A\nabla)g_2(x)dx\\
=&\int_{\R^{2d}}g_1(x)K(x-y)g_2(y)dydx
\end{aligned}
\]
with $K=\nabla\cdot \A\nabla(\nabla \cdot \A\nabla \G)$. By an explicit calculation, we have, for $x \neq 0$,
 \begin{equation}
 \label{e.comput.K}
\begin{aligned}
K(x)/c=&\frac{1}{|x|^{d+2}}[-d\tr(\A)^2-2d\tr(\A^2)]+\frac{x^t\A x \tr(\A)}{|x|^{d+4}}2d(d+2)\\
&+\frac{x^t\A^2x}{|x|^{d+4}}4d(d+2)-\frac{(x^t\A x)^2}{|x|^{d+6}}d(d+2)(d+4)
\end{aligned}
\end{equation}
%\[
%-\nabla \cdot  \A\nabla \G(x)=c\frac{\sum_{i,j} \A_{ij}(\delta_{ij}|x|^2-x_ix_jd)}{|x|^{d+2}}
%\]
for some constant $c\neq 0$. We show in Appendix~\ref{ap.matrix} that if $\A$ is not a multiple of the identity, then $(x^t\A x)^2/|x|^2$ cannot be a quadratic form. As a consequence, the function $x \mapsto |x|^{d+6} K(x)$ is a non-zero polynomial over $\R^d \setminus \{0\}$. 
%By considering the last term in the above expression, it is clear that if $K(x)\equiv 0$, then $(x^t\A x)^2/|x|^2$ is a quadratic form. By Lemma~\ref{lem:matrix}, this is true only when $\A$ is a multiple of identity. Therefore, $K(x)\not\equiv 0$, %there exists $x\neq 0$ such that $K(x)\neq 0$,
This implies that we can find $g_1,g_2$ whose supports are in $Q^\frac12 U$ and $Q^\frac12\ov U^c$ respectively and satisfy 
\[
\int_{\R^{2d}}g_1(x)K(x-y)g_2(y)dydx\neq 0.
\]
The proof is complete.
\end{proof}
\begin{proof}[Proof of Theorem~\ref{t.non-markov}]
The result follows from Lemmas~\ref{l.local} and \ref{l.non-loc}.
\end{proof}

\subsection{A counter-example with less regularity}
\label{ss.counter}
For the reader's convenience, we briefly recall the counter-example given in \cite{kotpre}. Let $X$ be the random distribution over $\R$ such that for every $f \in C^\infty_c$, $X(f)$ is a centered Gaussian with variance $\int_{\R} \Ll[f^2 + (f')^2\Rr]$, where $f'$ is the derivative of $f$. We let
$\LL^{-1} f := f - f''$,
so that $\LL$ is a convolution operator with kernel $\frac 1 2 \exp(-|x-y|)$, and hence is not local. 

\smallskip

The set of admissible test functions for $X$ can be extended to every element of the Sobolev space $W^{1,2}$ of functions in $L^2$ with weak derivative in $L^2$. 
We recall that by Morrey's inequality, elements of $W^{1,2}$ are continuous functions. For every open or closed $A \subset \R$, we write 
$$
W^{1,2}(A) := \Ll\{f \in W^{1,2} \ : \ f = 0 \text{ on } A^c\Rr\}.
$$
Note that if $A$ is closed, then
$$
W^{1,2}(A) = \bigcap_{\text{open } U \supseteq A} W^{1,2}(U).
$$
The space $W^{1,2}$ has a Hilbert space structure. We write $W_\msf{perp}^{1,2}(A)$ for the orthogonal complement of $W^{1,2}(A)$ in $W^{1,2}$. 
For every $V \subset W^{1,2}$, we write
$$
X(V) := \{X(f) \ : \ f \in V\},
$$ 
and check that for every open or closed $A \subset \R^d$,
$$
\mcl H(A) = X\Ll(W^{1,2}(A)\Rr).
$$
%\gucomment{ is $\mcl H(A) = X\Ll(W^{1,2}(A)\Rr)$ obvious? Here $A$ is closed?}
For $A$ open or closed, denote by $\mcl F_H(A)$ the $\sigma$-algebra generated by the random variables in $\mcl H(A)$. We claim that for every such $A$, we have $\mcl F(A) = \mcl F_H(A)$. The identity is clear if $A$ is open. For $A$ closed, we clearly have $\mcl F_H(A) \subset \mcl F(A)$. Conversely, observe that 
$$
W^{1,2}(A) + \bigcup_{\text{open } U \supseteq A} W^{1,2}_\msf{perp}(U)
$$
is dense in $W^{1,2}$ (the closure of the union above is $W^{1,2}_\msf{perp}(A)$). Hence, by the martingale convergence theorem, if $Z$ is a bounded $\mcl F$-measurable random variable, then 
\begin{equation}
\label{e.lim.Z}
Z = \lim_{\text{open } U \downarrow A} \E\Ll[Z \, | \, X \Ll(W^{1,2}(A) + W^{1,2}_\msf{perp}(U)\Rr)\Rr].
\end{equation}
If $Z$ is $\mcl F(A)$-measurable, then it is independent of $X(W^{1,2}_\msf{perp}(U))$. By \eqref{e.lim.Z}, it is therefore $\mcl F_H(A)$-measurable. 

Let $f \in W^{1,2}((-\infty,0])$ and $g \in W^{1,2}([0,+\infty))$. We have
\begin{align*}
\E[X(f) X(g)] & = \int_\R \Ll[ fg + f' g' \Rr] ,
\end{align*}
where $f', g' \in L^2(\R)$ are the weak derivatives of $f,g \in L^2(\R)$. The functions $f$ and $f'$ are supported in $(-\infty,0]$, while $g$ and $g'$ are supported in $[0,+\infty)$, so $\E[X(f) X(g)] = 0$. Since the spaces $\mcl H((-\infty,0])$ and $\mcl H([0,+\infty))$ are Gaussian, we infer that $\mcl F((-\infty,0])$ and $\mcl F([0,+\infty))$ are independent. Moreover, $\mcl F(\{0\})$ is trivial, and therefore the field $X$ is Markovian with respect to $(-\infty,0)$.

\smallskip

Note that the field $X$ has the regularity of the derivative of white noise. The proof of Step 3 of Lemma~\ref{l.decompH} breaks down, and 
$$
\mcl H((-\infty,0]) + \mcl H([0,+\infty)) \neq \mcl H,
$$
since $W^{1,2}((-\infty,0]) + W^{1,2}([0,+\infty))$ only contains functions that vanish at the origin, and therefore
$$
W^{1,2}((-\infty,0]) + W^{1,2}([0,+\infty)) \neq W^{1,2}.
$$ 
\section{Homogenization and expansion of effective fluctuation tensor}
\label{s.homog}

We now turn to stochastic homogenization. We focus on discrete elliptic equations in divergence form:
\[
\nabla^* \a(x)\nabla u(x)=f(x), \ \ x\in \Z^d,
\]
with $d \ge 3$, where $\a:\Z^d\to \msf{Sym}_d^+$ is a field of diagonal matrices. The entries $\a_{ii}(x), x\in \Z^d, i=1,\ldots,d$ are i.i.d.\ random variables defined on the probability space $(\Omega,\msf{F},\mathbb{P})$. We assume that $\a_{ii}(x) \in [1-\tau_0,1+\tau_0]$ for some fixed $\tau_0 \in(0,1)$. Let $\B$ be the set of nearest neighbor edges in $\Z^d$, $\{e_i,i=1,\ldots,d\}$ be the canonical basis, and for $e=(x,x+e_i)\in\B$, we view
\[
\a(e):=\a_{ii}(x)
\]
 as the random conductance on the edge linking $x$ to $x+e_i$. The discrete gradient and divergence for $f:\Z^d\to \R$ and $F:\Z^d\to \R^d$ are defined respectively by
\[
\nabla f=(\nabla_1f,\ldots,\nabla_d f), \ \ \nabla^*F=\sum_{i=1}^d \nabla_i^* F_i,
\]
with 
\[
\nabla_i f(x)=f(x+e_i)-f(x), \ \ \nabla_i^* F_i=F_i(x-e_i)-F_i(x).
\]
For any $e=(\underline{e},\bar{e})\in \B$ and $f:\Z^d\to \R$, we also write $\nabla f(e)=f(\bar{e})-f(\underline{e})$. For any $\xi\in \R^d$ and $i=1,\ldots,d$, we define $\xi(e_i)=\xi_i$.

There exists a deterministic matrix $\ah \in \mathsf{Sym}_d^+$ such that if $\alpha>0$ and $u^{(\eps)}$ solves
\[
(\eps^2\alpha +\nabla^* \a(x)\nabla) u^{(\eps)}(x)=\eps^2 f(\eps x), \qquad x\in \Z^d,
\]
then it was shown in \cite[Theorem 4]{kuen} that $u_\eps(x):=u^{(\eps)}(x/\eps)$ $(x\in \eps \Z^d)$ converges in $L^2(\R^d\times \Omega)$ to the deterministic solution of 
\[
(\alpha-\nabla\cdot \ah\nabla) u_\h(x)=f(x),  \qquad x\in \R^d,
\]
where $u_\eps$ is extended to $\R^d$ as a piecewise constant function. In other words, on large scales, the random coefficients $\a(x)$ behave like the homogeneous, deterministic coefficients $\ah$, and the discrete heterogeneous operator $\nabla^* \a(x)\nabla$ is averaged as the continuous homogeneous operator $-\nabla\cdot \ah \nabla$.

As was pointed out in the introduction, the corrector plays an important role in proving the convergence of $u_\eps\to u_\h$ and calculating the homogenized matrix $\ah$. The equation of the corrector in the direction of $\xi\in \R^d$ says
\begin{equation}
\nabla^* \a(x)(\nabla\phi_\xi(x)+\xi)=0, \qquad x\in \Z^d.
\label{eq:coreq}
\end{equation}
Since we assume $d\geq 3$, there exists a stationary zero-mean random field solving \eqref{eq:coreq}, as was shown in~\cite{gloria2011optimal}. 

From now on, we assume furthermore that 
\[
\a(e)=a(\zeta_e),
\] 
where $a:\R\to \R$ is a fixed twice differentiable function with bounded first and second derivatives (and taking values in $[1-\tau_0,1+\tau_0]$), and $\{\zeta_e:e\in\B\}$ is a sequence of i.i.d.\ standard Gaussian random variables. Under this technically convenient condition, the large-scale behavior of the corrector is approximately a generalized GFF. More precisely, it was shown in \cite{mourrat2014correlation, MN} that 
\begin{equation}
\label{e.conv}
\eps^{-\Ll(\frac d 2 - 1\Rr)} \phi_\xi(\, \cdot/\eps) \xrightarrow[\eps \to 0]{\text{(law)}} \Phi_\xi,
\end{equation}
where $\Phi_\xi$ solves
\begin{equation}
-\nabla \cdot \ah \nabla \Phi_\xi = \nabla \cdot W_\xi
\label{eq:corlim}
\end{equation}
for a certain white noise vector field $W_\xi$. A  heuristic derivation of \eqref{eq:corlim} can be found in \cite[Section 1]{gu-mourrat}.

The Gaussian white noise vector field $W_\xi$ appearing in \eqref{eq:corlim} has a covariance matrix $\Qh_\xi$ given explicitly by
\begin{multline}
[\Qh_\xi]_{ij}
= \sum_{k=1}^d \big\la (e_i+\nabla\phi_i)(e_k)(\xi+\nabla\phi_\xi)(e_k) \a'(e_k) \\ 
(1+\L)^{-1} \a'(e_k)(e_j+\nabla\phi_j)(e_k)(\xi+\nabla\phi_\xi)(e_k) \big\ra.
\label{eq:qij}
\end{multline}
Here $\la\cdot\ra$ denotes the expectation in $\Omega$ and $\{\phi_i=\phi_{e_i}, i=1,\ldots,d\}$ are the correctors in the canonical directions. For sufficiently smooth functions on $\Omega$, the weak derivative with respect to $\zeta_e$ is denoted by $\partial_e$, and its adjoint is $\partial_e^*=\zeta_e-\partial_e$. In \eqref{eq:qij}, we have set
\[
\L:=\sum_{e\in\B}\partial_e^*\partial_e,
\]
and
\[
\a'(e):=\partial_e  \a(e)=a'(\zeta_e).
\]
Note that the map $\xi \mapsto \Qh_\xi$ is quadratic. This object should thus be viewed as a four-fold tensor, which we propose to call the \emph{effective fluctuation tensor}. 

The convergence of the random distribution in \eqref{eq:coreq} to $\Phi_\xi$ motivates our study of generalized GFFs. Since we assume the coefficients to be i.i.d.,\ the homogenized matrix inherits the symmetries of the lattice, and $\ah$ is thus a multiple of the identity. %Therefore, in view of the results of the previous section, the question of whether $\Phi_\xi$ satisfies the Markov property boils down to whether $\Qh_\xi$ is a multiple of the identity. 

\smallskip 

The goal of this section is to obtain an expansion of $\Qh_\xi$ in the regime of small ellipticity contrast. This is interesting per se, e.g.\ as a means to compute the effective fluctuation tensor when we expect only a small amount of random fluctuation in the underlying medium, in which case $\Qh_\xi$ may be replaced by the series expansion up to certain order depending on the desired accuracy. It also allows us to give examples of environments such that $\Qh_\xi$ is not a multiple of the identity, which implies, in view of Theorem~\ref{t.non-markov}, that the limiting field $\Phi_\xi$ does \emph{not} satisfy the Markov property. 

Without loss of generality, we may assume that $\a$ takes the form
\begin{equation}
\label{e.struc-a}
\a(x)=\Id+\tau \b(x)
\end{equation}
where $\b$ is a field of diagonal matrices with $\la\b\ra=0$ and $\|\b\|_{L^\infty} \le 1$, and $\tau \in (0,1)$ is a free parameter that we will take sufficiently small. We let $\delta_{ij}$ be the Kronecker symbol, that is, $\delta_{ij} = 1$ if $i = j$, and $\delta_{ij} = 0$ otherwise. 
\begin{thm}
\label{t.expansion}
There exists $\tau_0 > 0$ such that for every $\tau \in [0,\tau_0)$,  $\xi\in \R^d$ and $i,j=1,\ldots,d$, we have the convergent expansion
\[
[\Qh_\xi]_{ij}=\tau^2\sum_{l=0}^\infty c_{\xi,l,i,j}\tau^l,
\]
where $c_{\xi,l,i,j}$ can be computed explicitly as explained below, and
\begin{equation}
\label{c0}
c_{\xi,0,i,j} = \delta_{ij} \xi_i^2 \, \langle \b(e)^2 \rangle ,
\end{equation}
\begin{equation}
\label{c1}
i \neq j \quad \implies \quad c_{\xi,1,i,j} = 0,
\end{equation}
\begin{equation}
\label{c2}
i \neq j \ \mbox{ and } \xi_i \xi_j \neq 0 \quad \implies \quad c_{\xi,2,i,j} \neq 0.
\end{equation}
In particular, if $\tau >0$ is sufficiently small and $\xi \neq 0$, then $\Qh_\xi$ is not a multiple of identity.
\end{thm}

We first perform a formal expansion of $\Qh_\xi$ and observe that \eqref{c0}--\eqref{c2} should hold. We then justify the full expansion rigorously. 

\subsection{A formal expansion}
\label{s:forex}

%In the small ellipticity regime, we may assume 
%\[
%\a(x)=\Id+\tau \b(x)
%\]
% with $0<\tau \ll 1$ and $\b$ is some bounded diagonal matrix with $\la\b\ra=0$. 
In view of \eqref{e.struc-a}, the corrector equation \eqref{eq:coreq} in the direction $\eta\in \R^d$ can be rewritten as
\[
-\Delta \phi_\eta=-\tau\nabla^* \b(\nabla\phi_\eta+\eta),
\]
where $\Delta:=-\nabla^*\nabla$ is the discrete Laplacian. Formally, we can write
\[
\nabla\phi_\eta=-\tau \nabla (-\Delta)^{-1}\nabla^* \b(\nabla\phi_\eta+\eta).
\]
If the operator $-\tau \nabla(-\Delta)^{-1}\nabla^* \b$ is a contraction (note that $\tau\ll1$), then
\begin{equation}
\eta+\nabla\phi_\eta=\sum_{k=0}^\infty X_{k,\eta},
\label{eq:exphi}
\end{equation}
 with 
\begin{equation}
X_{k,\eta}=[-\tau \nabla(-\Delta)^{-1}\nabla^*\b]^k \eta.
\label{eq:forex}
\end{equation}
There are four factors of $\eta+\nabla\phi_\eta$ appearing in \eqref{eq:qij} with $\eta=e_i,e_j,\xi$, which we will replace by the expansion in \eqref{eq:exphi}. Since $\a'(e)=\tau \b'(e)$ for any $e\in \B$, we obtain an expansion of $[\Qh_\xi]_{ij}$ in terms of~$\tau$ written as:
\begin{equation}
[\Qh_\xi]_{ij}=\tau^2\sum_{l=0}^\infty c_{\xi,l}\tau^l,
\end{equation}
where we suppressed the dependence of $c_{\xi,l}$ on $i,j$ in the notation.
We emphasize that $\sum_{l=0}^\infty c_{\xi,l}\tau^l$ is an expansion of 
\begin{multline}
\frac{1}{\tau^2}[\Qh_\xi]_{ij}
= \sum_{k=1}^d \big\la (e_i+\nabla\phi_i)(e_k)(\xi+\nabla\phi_\xi)(e_k) \b'(e_k) \\ 
(1+\L)^{-1} \b'(e_k)(e_j+\nabla\phi_j)(e_k)(\xi+\nabla\phi_\xi)(e_k) \big\ra,
\end{multline}
where the only $\tau-$dependent terms are those factors of $\eta+\nabla\phi_\eta$ with $\eta=e_i,e_j,\xi$.
%By \eqref{eq:qij} and noting that $\a'(e)=\tau \b'(e)$ for any $e\in \B$, we automatically have that $c_{\xi,l}=0$ for $l=0,1$. Thus
%\[
%[\Qh_\xi]_{ij}=\sum_{l=2}^\infty c_{\xi,l}\tau^l.
%\]
%
%It turns out that if we consider the off-diagonal term $[\Qh_\xi]_{ij}$, we will be able to read off the expansion coefficients $c_{\xi,k}$ that $\Qh_\xi$ is not a multiple of the identity. 
The following are simple calculations using the i.i.d.\  structure of the random coefficients.

We introduce some notation. Let $G(x,y)$ be the Green function of the discrete Laplacian $-\Delta$, and $\nabla\nabla G(x,y)$ be the Hessian matrix such that 
\[
[\nabla\nabla G(x,y)]_{ij}=\nabla_{i}\nabla_{j}G(x,y),
\]
where $\nabla_{i},\nabla_{j}$ are with respect to the $x,y$ variable respectively. For any edge $e=(\underline{e},\bar{e})\in \B$, we also write
\[
\nabla \nabla G(e,y)=\nabla G(\bar{e},y)-\nabla G(\underline{e},y).
\]

\medskip
\textbf{Second order.} To get the second order term, we replace all factors in \eqref{eq:qij} of the form $\eta+\nabla\phi_\eta$ by $X_{0,\eta}=\eta$ to obtain
\[
\begin{aligned}
c_{\xi,0}=&\sum_{k=1}^d \la \delta_{ik}\xi_k\b'(e_k) (1+\L)^{-1}\b'(e_k)\delta_{jk}\xi_k\ra\\
=&\la \b'(e)(1+\L)^{-1}\b'(e)\ra \sum_{k=1}^d  \delta_{ik}\delta_{jk}\xi_k^2=\delta_{ij} \xi_i^2 \, \langle \b(e)^2 \rangle,
\end{aligned}
\]
where we used \cite[Proposition~3.1]{mourrat2014correlation} in the last step. 
%thus
%\[
%[\Qh_\xi]_{ij}=16\la a'(e),(1+\L)^{-1}a'(e)\ra \xi_i^2\delta_{ij}+O(\eta),
%\]
%which implies that in order for $\Qh_\xi$ to be a multiple of identity, we need to assume $\xi_i^2=\xi_j^2$ for all $i\neq j$.

\medskip

\textbf{Third order.} From now on, we focus on the case $i \neq j$. To get the third order term, we replace one of the four factors of $\eta+\nabla\phi_\eta$ in \eqref{eq:qij} by $X_{1,\eta}$ and all other three by $X_{0,\eta}=\eta$. We write
\[
\begin{aligned}
X_{1,\eta}(x)&=[-\tau \nabla(-\Delta)^{-1}\nabla^* \b \eta](x)\\
&=-\tau \sum_{y\in \Z^d} \nabla\nabla G(x,y) \b(y)\eta,
\end{aligned}
\]
which is a vector for any $x\in \Z^d$. We also write
\[
X_{1,\eta}(e)=-\tau \sum_{y\in \Z^d} \nabla\nabla G(e,y)\cdot \b(y)\eta,
\]
which is a scalar for any $e\in \B$. We claim that we only need to consider $e_i+\nabla\phi_i$ or $e_j+\nabla\phi_j$. Otherwise if we replace $\xi+\nabla\phi_\xi$ by $X_{1,\xi}$, i.e., in \eqref{eq:qij}
\[
e_i(e_k)+\nabla \phi_i(e_k)\mapsto  e_i(e_k), \ \ e_j(e_k)+\nabla \phi_j(e_k)\mapsto  e_j(e_k),
\]
and
\[
\xi(e_k)+\nabla\phi_\xi(e_k)\mapsto -\tau \sum_{y\in \Z^d} \nabla\nabla G(e_k,y)\cdot \b(y)\xi,
\]
we observe that there is a factor of $e_i(e_k)e_j(e_k)=\delta_{ik}\delta_{jk}=0$ since $i\neq j$. Therefore, we have
\[
\begin{aligned}
c_{\xi,1}=&-\xi_k^2\sum_{k=1}^d \langle \sum_{y\in \Z^d} \nabla\nabla G(e_k,y)\cdot \b(y)e_i\b'(e_k) (1+\L)^{-1}\b'(e_k)\delta_{jk}\ra\\
&-\xi_k^2\sum_{k=1}^d \langle \b'(e_k)\delta_{ik} (1+\L)^{-1}\b'(e_k) \sum_{y\in \Z^d} \nabla\nabla G(e_k,y)\cdot \b(y)e_j\ra.
\end{aligned}
\]
For the first term on the r.h.s. of the above expression, we can write
\[
\begin{aligned}
&\la \nabla\nabla G(e_k,y)\cdot \b(y)e_i\b'(e_k) (1+\L)^{-1}\b'(e_k)\delta_{jk}\ra\\
=&\la\nabla \nabla G(e_j,y)\cdot \b(y)e_i \b'(e_j)(1+\L)^{-1}\b'(e_j)\ra \delta_{jk}.
\end{aligned}
\]
Since $i\neq j$, it is clear that $\nabla \nabla G(e_j,y)\cdot \b(y)e_i$ and $\b'(e_j)(1+\L)^{-1}\b'(e_j)$ are independent for all $y\in \Z^d$,  so we have the expectation equals zero because $\b$ has mean zero. The same discussion applies to the second term, thus
\[
c_{\xi,1}=0.
\]

\medskip
\textbf{Fourth order.} To get the fourth order term, we have multiple options. First we consider the case when we have one factor of $X_{2,\eta}$ coming from $\eta+\nabla\phi_\eta$. By the same reason as before, we can not choose $\eta=\xi$, otherwise we have a factor of $\delta_{ik}\delta_{jk}=0$. If instead we consider $\eta=e_i$ (the same discussion applies to $e_j$) and write:
\[
X_{2,e_i}(x)=\tau^2\sum_{y,z\in \Z^d}\nabla\nabla G(x,y)\b(y)\nabla\nabla G(y,z)\b(z)e_i
\]
for $x\in \Z^d$ 
or 
\[
X_{2,e_i}(e)=\tau^2\sum_{y,z\in \Z^d}\nabla\nabla G(e,y)\cdot \b(y)\nabla\nabla G(y,z)\b(z)e_i
\]
for $e\in \B$, then the contribution to $[\Qh_\xi]_{ij}$ is
\[
\xi_j^2\tau^4 \la \sum_{y,z\in \Z^d}\nabla\nabla G(e_j,y)\cdot \b(y)\nabla\nabla G(y,z)\b(z)e_i\b'(e_j)(1+\L)^{-1}\b'(e_j)\ra.
\]
By the same discussion as before, $\b(z)e_i$ is independent of $\b'(e_j)(1+\L)^{-1}\b'(e_j)$ for all $z\in \Z^d$ because $i\neq j$. Since $\nabla\nabla G(e_j,y)\cdot\b(y)\nabla\nabla G(y,z)\b(z)e_i$ is a linear combination of $\b_{kk}(y)$ and $\b_{ii}(z)$, the only nonzero contribution after taking the expectation is when $y=z$ and $k=i$:
\[ 
\begin{aligned}
&\xi_j^2\tau^4 \la \sum_{y,z\in \Z^d}\nabla\nabla G(e_j,y)\cdot \b(y)\nabla\nabla G(y,z)\b(z)e_i \b'(e_j)(1+\L)^{-1}\b'(e_j)\ra\\
=&\xi_j^2\tau^4\nabla_i\nabla_iG(0,0)\sum_{y\in \Z^d} \nabla\nabla_i G(e_j,y)\la \b_i^2(y) \b'(e_j)(1+\L)^{-1}\b'(e_j)\ra\\
=&\xi_j^2\tau^4\nabla_i\nabla_iG(0,0)\sum_{y\in \Z^d} \nabla\nabla_i G(e_j,y)\la \b^2(e_i)\ra \la \b'(e_j)(1+\L)^{-1}\b'(e_j)\ra
\end{aligned}
\]
Since $\nabla\nabla_iG(e_j,y)$ is a gradient, it follows that $\sum_{y\in \Z^d} \nabla\nabla_i G(e_j,y)=0$.

Now we consider the case when we have two factors of $X_{1,\eta}$. Recall that 
%\begin{equation*}
%\begin{aligned}
%&[\Qh_\xi]_{ij}\\
%=&\sum_{k=1}^d \la (e_i+\nabla\phi_i)(e_k)(\xi+\nabla\phi_\xi)(e_k) \a'(e_k)(1+\L)^{-1} \a'(e_k)(e_j+\nabla\phi_j)(e_k)(\xi+\nabla\phi_\xi)(e_k)\ra.
%\end{aligned}
%\end{equation*}
\begin{multline*}
[\Qh_\xi]_{ij}
= \sum_{k=1}^d \big\la (e_i+\nabla\phi_i)(e_k)(\xi+\nabla\phi_\xi)(e_k) \a'(e_k) \\ 
(1+\L)^{-1} \a'(e_k)(e_j+\nabla\phi_j)(e_k)(\xi+\nabla\phi_\xi)(e_k) \big\ra.
%\label{eq:qij}
\end{multline*}
By symmetry, we only need to consider the following cases:
\[
\begin{aligned}
(i)&= \sum_{k=1}^d \la X_{1,e_i}(e_k) X_{1,\xi}(e_k)\a'(e_k)(1+\L)^{-1} \a'(e_k)\ra \delta_{jk}\xi_k, \\
(ii)&=\sum_{k=1}^d \la X_{1,e_i}(e_k)\a'(e_k) (1+\L)^{-1} \a'(e_k)X_{1,e_j}(e_k)\ra \xi_k^2,\\
(iii)&=\sum_{k=1}^d \la X_{1,e_i}(e_k)\a'(e_k) (1+\L)^{-1} \a'(e_k)X_{1,\xi}(e_k)\ra \delta_{jk}\xi_k.
\end{aligned}
\]
Note that we do not consider replacing both factors of $\xi+\nabla\phi_\xi$ since it leads to a factor of $\delta_{ik}\delta_{jk}=0$.

For $(i)$, we write
\[
X_{1,e_i}(e_k)X_{1,\xi}(e_k)=\tau^2 \sum_{y\in \Z^d}\sum_{z\in \Z^d}(\nabla\nabla G(e_k,y)\cdot \b(y)e_i)(\nabla\nabla G(e_k,z)\cdot \b(z)\xi),
\]
and in order to get a nonzero contribution in $(i)$, we only need the terms with $y=z$ and the factor $\xi_i$, i.e.,
\[
\tau^2\sum_{y\in \Z^d} |\nabla\nabla_i G(e_k,y)|^2 |\b_{ii}(y)|^2\xi_i,
\]
which implies
\[
\begin{aligned}
(i)=&\tau^4 \xi_i\xi_j\sum_{y\in \Z^d}  |\nabla\nabla_i G(e_j,y)|^2  \la |\b_{ii}(y)|^2\b'(e_j)(1+\L)^{-1}\b'(e_j)\ra
\\
=&\tau^4 \xi_i\xi_j\sum_{y\in \Z^d}  |\nabla\nabla_i G(e_j,y)|^2  \la |\b(e_i)|^2\ra\la \b'(e_j)(1+\L)^{-1}\b'(e_j)\ra.
\end{aligned}
\]
In the last step, we used the fact that $\a'(e_j)(1+\L)^{-1}\a'(e_j)$ is independent of $\b_{ii}(y)$ for all $y\in \Z^d$.

For $(ii)$, we write
\[
\begin{aligned}
&\la X_{1,e_i}(e_k)\a'(e_k)(1+\L)^{-1} \a'(e_k) X_{1,e_j}(e_k)\ra\\
=&\tau^4\sum_{y,z\in \Z^d}\la \nabla\nabla G(e_k,y)\cdot \b(y)e_i \b'(e_k)(1+\L)^{-1}\b'(e_k)\nabla\nabla G(e_k,z)\cdot\b(z)e_j\ra,
\end{aligned}
\]
and it is clear that for any $k=1,\ldots,d$, we have an independent factor of $\b(y)e_i$ or $\b(z)e_j$ since $i\neq j$, thus the above expression equals to zero.

For $(iii)$, we write 
\[
\begin{aligned}
&\la X_{1,e_i}(e_j)\a'(e_j) (1+\L)^{-1}\a'(e_j)X_{1,\xi}(e_j)\ra\\
=&\tau^4\sum_{y,z\in \Z^d}\la \nabla\nabla G(e_j,y)\cdot \b(y)e_i \b'(e_j)(1+\L)^{-1}\b'(e_j)\nabla \nabla G(e_j,z)\cdot\b(z)\xi\ra.
\end{aligned}
\]
Similarly, we need the terms with $y=z$ and the factor $\xi_i$, which gives
\[
\begin{aligned}
&\la X_{1,e_i}(e_j)\a'(e_j)(1+\L)^{-1}\a'(e_j)X_{1,\xi}(e_j)\ra\\
=&\xi_i\tau^4\sum_{y\in \Z^d}\la \nabla\nabla G(e_j,y)\cdot \b(y)e_i \b'(e_j)(1+\L)^{-1}\b'(e_j)\nabla \nabla G(e_j,y)\cdot\b(y)e_i\ra,
\end{aligned}
\]
so
\[
\begin{aligned}
(iii)=&\xi_i\xi_j\tau^4\sum_{y\in \Z^d} \la \nabla\nabla G(e_j,y)\cdot \b(y)e_i \b'(e_j)(1+\L)^{-1}\b'(e_j)\nabla \nabla G(e_j,y)\cdot \b(y)e_i\ra\\
=&\xi_i\xi_j\tau^4\sum_{y\in \Z^d} |\nabla\nabla_i G(e_j,y)|^2 \la \b_{ii}(y) \b'(e_j)(1+\L)^{-1}\b'(e_j)\b_{ii}(y)\ra\\
=&\xi_i\xi_j\tau^4\sum_{y\in \Z^d} |\nabla\nabla_i G(e_j,y)|^2   \la \b(e_i)\b'(e_j) (1+\L)^{-1} \b'(e_j)\b(e_i)\ra.
\end{aligned}
\]

To summarize, the above formal calculation shows that for $i \neq j$,
\begin{equation}
\label{e.c01}
c_{\xi,0}=c_{\xi,1}=0,
\end{equation}
and
\begin{equation}
\begin{aligned}
c_{\xi,2}= &\xi_i\xi_j\sum_{y\in \Z^d}  |\nabla\nabla_i G(e_j,y)|^2  \la |\b(e_i)|^2\ra\la \b'(e_j)(1+\L)^{-1}\b'(e_j)\ra\\
&+\xi_i\xi_j\sum_{y\in \Z^d}  |\nabla\nabla_i G(e_j,y)|^2 \la \b(e_i)\b'(e_j) (1+\L)^{-1} \b'(e_j)\b(e_i)\ra,
\end{aligned}
\label{eq:c2}
\end{equation}
which is non-zero if $\xi_i\xi_j\neq 0$. %In particular, if $\xi_i \xi_j \neq 0$ and $\tau > 0$ is sufficiently small, then $\Qh_\xi$ is not a multiple of identity.

\subsection{Proof of the expansion.}

Before proving the expansion rigorously, we introduce more notation. Let $\zeta=(\zeta_e)_{e\in\B}\in \Omega$ denote the sample point, and for $x\in \Z^d$, we define the shift operator $\tau_x$ on $\Omega$ by $(\tau_x \zeta)_e=\zeta_{x+e}$, where $x+e:=(x+\underline{e},x+\bar{e})$ is the edge obtained by shifting $e$ by $x$. Since $\{\zeta_e\}_{e\in\mathbb{B}}$ are i.i.d., $\{\tau_x\}_{x\in \Z^d}$ is a group of measure-preserving transformations. 
%With any measurable function $f:\Omega\to \R$, we can associate a stationary random field $\tilde{f}(x,\zeta)$ defined by
%\[
%\tilde{f}(x,\zeta)=(T_xf)(\zeta)=f(\tau_x\zeta).
%\]
We can define the operator 
\[
T_x f(\zeta)=f(\tau_x\zeta)
\]
for any measurable function $f$ on $\Omega$, and the generators of $T_x$, denoted by $\{D_i\}_{i=1}^d$, are defined by $D_if:=T_{e_i}f-f$. The adjoint $D_i^*$ is defined by $D_i^* f:=T_{-e_i}f-f$. We denote the gradient on $\Omega$ by $D=(D_1,\ldots,D_d)$ and the divergence $D^*F:=\sum_{i=1}^d D_i^*F_i$ for $F:\Omega\to \R^d$. The norm in $L^p(\Omega)$ is denoted by $\|\cdot\|_p$.

Now we can formulate the corrector equation 
\begin{equation}
\nabla^* \a(x)(\nabla\phi_\eta(x)+\eta)=0
\label{eq:newcoeq}
\end{equation}
 in the probability space as
 \begin{equation}
 D^*\a(D\phi_\eta+\eta)=0.
 \label{eq:procoeq}
 \end{equation}
%We can view $\phi_\eta$ as function on $\Z^d\times \Omega$ and notice that if $\phi_\eta(0,\zeta)$ satisfies \eqref{eq:procoeq} almost surely in $\Omega$, then for $\mathbb{P}-$a.e. $\zeta$, $\phi_\eta(x,\zeta):=\phi_\eta(\tau_x\zeta)$ is a solution to \eqref{eq:newcoeq}. 
We note that \eqref{eq:procoeq} holds almost surely in $\Omega$ and \eqref{eq:newcoeq} holds on $\Z^d$ for $\mathbb{P}-$a.e. $\zeta$. From now on, with an abuse of notation, we regard $\a,\phi_\eta$ as functions on $\Omega$ with
\[
\a(\zeta)=\mathrm{diag}(\{a(\zeta_{e_i})\}_{i=1,\ldots,d}).
\]
It is clear that if $\zeta \mapsto \phi_\eta(\zeta)$ solves \eqref{eq:procoeq}, then $x \mapsto \phi_\eta(\tau_x\zeta)$ solves \eqref{eq:newcoeq}. 

Since $\a=\Id+\tau\b$, \eqref{eq:procoeq} can be written as
\[
-\Delta \phi_\eta=-\tau D^*\b(D\phi_\eta+\eta)
\]
with $-\Delta=D^*D$ the Laplacian on the probability space, so formally
\[
D\phi_\eta=-\tau D(-\Delta)^{-1}D^*\b(D\phi_\eta+\eta).
\]
The operator $D(-\Delta)^{-1}D^*$ is defined by the following lemma.

\begin{lem}
Let $F=(F_i)_{i=1,\ldots,d}\in L^2(\Omega,\R^d)$, $\phi_\lambda:\Omega\to \R$ solve
\[
(\lambda-\Delta)\phi_\lambda=D^*F,
\]
and $\Psi=(\Psi_i)_{i=1,\ldots,d}=D(-\Delta)^{-1}D^*F$ be a weak limit of $D\phi_\lambda$ in $L^2(\Omega)$, then $\Psi$ is the unique function in $L^2(\Omega,\R^d)$ that satisfies the following three properties:

(i) $\la \Psi\ra=0$;

(ii) $D_i\Psi_j=D_j\Psi_i$ for all $i,j=1,\ldots,d$;

(iii) $D^*\Psi=D^*F$.

Furthermore, $\|\Psi\|_2\leq \|F\|_2$, and $D\phi_\lambda\to \Psi$ in $L^2(\Omega)$ as $\lambda\to0$.%$D(\lambda-\Delta)^{-1}D^*F\to D(-\Delta)^{-1}D^*F$ in $L^2(\Omega)$ as $\lambda\to 0$.
\label{lem:pro}
\end{lem}

\begin{proof}
It is a special case of \cite[Theorem 3]{kuen}. The uniqueness follows from \cite[Theorem 2]{papanicolaou1979boundary} with minor modifications. For the strong $L^2(\Omega)$ convergence, we only need to show the convergence of $\la D\phi_\lambda\cdot D\phi_\lambda\ra$. On one hand, we have
\[
\begin{aligned}
\lambda\la \phi_\lambda^2\ra+\la D\phi_\lambda\cdot D\phi_\lambda\ra=&\la F\cdot D\phi_\lambda\ra\to  \la F \cdot \Psi\ra=\la \Psi\cdot \Psi\ra,
\end{aligned}
\]
where the last step comes from $\la \Psi\cdot DG\ra=\la F\cdot DG\ra$ for any test function $G$. This implies
\[
\limsup_{\lambda\to 0}\la D\phi_\lambda\cdot D\phi_\lambda\ra\leq \la \Psi\cdot\Psi\ra.
\]
On the other hand, 
\[
\la \Psi\cdot \Psi\ra=\lim_{\lambda\to0}\la \Psi\cdot D\phi_\lambda\ra \leq \liminf_{\lambda\to 0}\|\Psi\|_2\|D\phi_\lambda\|_2.
\]
Thus $\la D\phi_\lambda\cdot D\phi_\lambda\ra\to \la \Psi\cdot \Psi\ra$ as $\lambda\to 0$. The proof is complete.
\end{proof}

%\begin{rem}
%By a spectral representation used in the proof of \cite[Theorem 3]{}, we actually have the strong convergence:
%\[
%D(\lambda-\Delta)^{-1}D^*F\to D(-\Delta)^{-1}D^*F
%\]
%in $L^2(\Omega)$ as $\lambda\to 0$.
%\end{rem}

To justify the expansion, we need $L^4$ rather than $L^2$ boundedness of the operator $D(-\Delta)^{-1}D^*$. The following lemma, whose proof is postponed to Appendix~\ref{ap.Lp}, is the key to justify the expansion.
\begin{lem}
For any $p\in(1,\infty)$, there exists $C_p>0$ such that
\[
\|D(\lambda-\Delta)^{-1}D^*F\|_p\leq C_p \|F\|_p
\] 
uniformly in $F\in L^p(\Omega,\R^d)$ and $\lambda\in [0,1]$.
\label{lem:bdp}
\end{lem}

\begin{proof}[Proof of Theorem~\ref{t.expansion}]
We first rewrite \eqref{eq:qij} as
%\begin{equation}
%\begin{aligned}
%&[\Qh_\xi]_{ij}\\
%=&\sum_{k=1}^d \la (\delta_{ik}+D_k\phi_i)(\xi_k+D_k\phi_\xi) \a'(e_k)(1+\L)^{-1} \a'(e_k)(\delta_{jk}+D_k\phi_j)(\xi_k+D_k\phi_\xi)\ra.
%\end{aligned}
%\label{eq:proqij}
%\end{equation}
\begin{multline}
[\Qh_\xi]_{ij}
= \tau^2\sum_{k=1}^d \big\la (\delta_{ik}+D_k\phi_i)(\xi_k+D_k\phi_\xi) \b'(e_k) \\ 
(1+\L)^{-1}\b'(e_k)(\delta_{jk}+D_k\phi_j)(\xi_k+D_k\phi_\xi) \big\ra.
\label{eq:proqij}
\end{multline}
Recall that $D\phi_\eta$ solves 
\[
(\Id-\tau\cP)D\phi_\eta=\tau\cP\eta,
\]
with 
\[
\cP:=- D(-\Delta)^{-1}D^*\b.
\]
By Lemma~\ref{lem:bdp} and the fact that $\b\in L^\infty(\Omega,\R^{d\times d})$, for any $p\in (1,\infty)$ and $n\in \N_+$ we have
\[
\|\cP^n \eta\|_p\leq C_p^n\|\eta\|_\infty
\]
for some constant $C_p>0$, so
\[
D\phi_\eta=(\tau\cP+\tau^2\cP^2+\ldots)\eta,
\]
where the convergence is in $L^p(\Omega,\R^d)$ for any $p\in (1,\infty)$, provided that $\tau<1/C_p$.
%which implies
%\[
%\|D\phi_\eta-\cP\eta-\cP^2\eta\|_p=O(\tau^3)
%\]
%as $\tau\to 0$. 
Therefore, using the fact that $(1+\L)^{-1}$ is bounded from $L^q(\Omega)$ to itself for any $q\geq 2$ \cite[Proposition 3.2]{mourrat2014correlation} and $\b'\in L^\infty(\Omega,\R^{d\times d})$, we can replace the factors $\eta_k+D_k\phi_\eta$ with $\eta=e_i,e_j,\xi$ in \eqref{eq:proqij} by the series
\[
\eta_k+D_k\phi_\eta= \sum_{n=0}^\infty \tau^n(\cP^n\eta)_k,%=\eta_k+\tau(\cP\eta)_k+\tau^2(\cP^2\eta)_k+\ldots.
\] 
%with an error of order $O(\tau^3)$. 
where $(\cP^n\eta)_k$ denotes the $k-$th component of $\cP^n\eta$ with $\cP^0=\Id$. Notice that we used Lemma~\ref{lem:bdp} for $p\geq 4$ and $\lambda=0$. In other words, we obtain an expansion of $\Qh_\xi$ in terms of $\tau$:
\[
[\Qh_\xi]_{ij}=\sum_{n_1,n_2,n_3,n_4=0}^\infty \tau^{2+n_1+n_2+n_3+n_4} Q_{n_1,n_2,n_3,n_4}(i,j,\xi)
\]
with 
\[
Q_{n_1,n_2,n_3,n_4}(i,j,\xi):=\sum_{k=1}^d\la (\cP^{n_1}e_i)_k(\cP^{n_2}\xi)_k\b'(e_k) (1+\L)^{-1}\b'(e_k)(\cP^{n_3}e_j)_k(\cP^{n_4}\xi)_k\ra.
\]
Thus, $[\Qh_\xi]_{ij}=\tau^2\sum_{l=0}^\infty c_{\xi,l}\tau^l$ with
\[
c_{\xi,l}=\sum_{n_1+\ldots+n_4=l}Q_{n_1,n_2,n_3,n_4}(i,j,\xi).
\]

To compute $c_{\xi,l}$ explicitly, we define
\[
\cP_\lambda:=- D(\lambda-\Delta)^{-1}D^*\b
\]
and
\begin{equation}
Q_{n_1,n_2,n_3,n_4}^\lambda(i,j,\xi):=\sum_{k=1}^d\la (\cP_\lambda^{n_1}e_i)_k(\cP_\lambda^{n_2}\xi)_k\b'(e_k) (1+\L)^{-1}\b'(e_k)(\cP_\lambda^{n_3}e_j)_k(\cP_\lambda^{n_4}\xi)_k\ra.
\label{eq:qlambda}
\end{equation}
 By Lemmas~\ref{lem:pro} and \ref{lem:bdp}, it follows that 
 \[
 Q_{n_1,n_2,n_3,n_4}^\lambda(i,j,\xi)\to Q_{n_1,n_2,n_3,n_4}(i,j,\xi)
\]
as $\lambda\to 0$, and we have
\begin{equation}
c_{\xi,l}=\sum_{n_1+\ldots+n_4=l}\lim_{\lambda\to 0}Q_{n_1,n_2,n_3,n_4}^\lambda(i,j,\xi).
\label{eq:cl}
\end{equation}
%
%
%\[
%[\Qh_\xi]_{ij}=\sum_{k=1}^d\la I_{i,k} \a'(e_k) (1+\L)^{-1} \a'(e_k)  I_{j,k}\ra+O(\tau^3),
%\]
%with 
%\[
%\begin{aligned}
%I_{i,k}&=[\delta_{ik}+(\cP e_i)_k+(\cP^2e_i)_k][\xi_k+(\cP \xi)_k+(\cP^2\xi)_k],\\
%I_{j,k}&=[\delta_{jk}+(\cP e_j)_k+(\cP^2e_j)_k][\xi_k+(\cP \xi)_k+(\cP^2\xi)_k].
%\end{aligned}
%\]
%We are left to compute 
%\[
%\sum_{k=1}^d\la I_{i,k} \a'(e_k) (1+\L)^{-1} \a'(e_k)  I_{j,k}\ra.
%\] The idea is to replace $\cP$ by 
%\[
%\cP_\lambda:=-\tau D(\lambda-\Delta)^{-1}D^*\b
%\]
%in the calculation and send $\lambda\to 0$. By Lemmas~\ref{lem:pro} and \ref{lem:bdp}, it is clear that 
%\[
%\la I_{i,k} \a'(e_k) (1+\L)^{-1} \a'(e_k)  I_{j,k}\ra=\lim_{\lambda\to 0}\la I_{i,k}^\lambda \a'(e_k) (1+\L)^{-1} \a'(e_k)  I_{j,k}^\lambda\ra,
%\]
%where $I_{i,k}^\lambda, I_{j,k}^\lambda$ are obtained by replacing the operator $\cP\mapsto \cP_\lambda$ in $I_{i,k},I_{j,k}$. 

With the mass regularization, we can write $\cP_\lambda^n$ in the physical domain
\begin{equation}
\begin{aligned}
\cP_\lambda^n\eta=& (-D(\lambda-\Delta)^{-1}D^*\b)^n\xi\\
=&(-1)^n\sum_{y_1,\ldots,y_n\in \Z^d}\left(\prod_{k=1}^n\nabla\nabla G_\lambda(y_{k-1},y_k)\b(\tau_{y_k}\zeta)\right)\eta,
\end{aligned}
\label{eq:plambda}
\end{equation}
with $y_0=0$ and $G_\lambda$ the Green function of $\lambda-\Delta$ on $\Z^d$. By plugging \eqref{eq:plambda} into \eqref{eq:qlambda} with $n=n_1,\ldots,n_4$ and $\eta=e_i,e_j,\xi$, the expectation can be computed explicitly using the i.i.d. structure of $\{\zeta_e\}_{e\in\B}$.

%Take $\cP_\lambda \xi$ for example:
%\[
%\begin{aligned}
%\cP_\lambda\xi=&- D(\lambda-\Delta)^{-1}D^*\b\xi\\
%=&- \sum_{y\in \Z^d} \nabla\nabla G_\lambda(0,y)\b(\tau_y\zeta)\xi,
%\end{aligned}
%\]
%with $G_\lambda$ the Green function of $\lambda-\Delta$ on $\Z^d$. 
%

To show that $\Qh_\xi$ may not be a multiple of identity, we only need to repeat the formal calculation in Section~\ref{s:forex} verbatim with $G$ replaced by $G_\lambda$ to derive that
\[
[\Qh_\xi]_{ij}=c_{\xi,2}\tau^4+O(\tau^5)
\]
with $c_{\xi,2}$ given by \eqref{eq:c2}.
%\[
%\lim_{\lambda\to 0}\sum_{k=1}^d\la I_{i,k}^\lambda \a'(e_k) (1+\L)^{-1} \a'(e_k)  I_{j,k}^\lambda\ra=c_{\xi,2}\tau^2+O(\tau^3).
%\]
The proof is complete.
\end{proof}

%\subsection{The projection is bounded from $L^p$ to $L^p$ $\forall p\geq 1$}

\appendix

\section{A linear algebra lemma}
\label{ap.matrix}

\begin{lem}
Let $\A$ be a symmetric, positive-definite matrix. The following statements are equivalent:
\begin{enumerate}
\item[(i)] there exists a symmetric matrix $\BB$ such that 
$$
\mbox{for all $x \in \R^n$, } (x^t \A x)^2=|x|^2x^t\BB x
$$ 

\item[(ii)] the matrix $\A$ is a multiple of identity.
\end{enumerate}
\label{lem:matrix}
\end{lem}

\begin{proof}
The implication (ii) $\implies$ (i) is obvious. We prove the converse implication by induction on the size $n$ of the matrix. Letting $A$ be an $n\times n$ matrix and assuming that the result holds for matrices of size $n-1$, we can choose $x=(x_1,\ldots,x_{n-1},0)$ to obtain $\A_{ij}=c\delta_{ij}$ and $\BB_{ij}=c^2\delta_{ij}$ for some constant $c>0$ with $i,j=1,\ldots,n-1$. Now, for any $i<n$, by considering the coefficients of $x_i^3x_n$ and $x_ix_n^3$, we have
\[
2c\A_{in}=\BB_{in}, \ \ 2\A_{nn}\A_{in}=\BB_{in},
\]
which implies
\begin{equation}
\A_{in}(\A_{nn}-c)=0.
\label{eq:ma1}
\end{equation}
By considering the coefficients of $x_i^2x_n^2$ and $x_n^4$, we have
\[
2c \A_{nn}+\A_{in}^2=c^2+\BB_{nn}, \ \ \A_{nn}^2=\BB_{nn},
\]
which implies
\begin{equation}
(\A_{nn}-c)^2=\A_{in}^2.
\label{eq:ma2}
\end{equation}
Combining \eqref{eq:ma1} and \eqref{eq:ma2}, we have $\A_{nn}=c, \A_{in}=0$. The proof is complete.
\end{proof}

\section{$D(\lambda-\Delta)^{-1}D^*$ is bounded from $L^p(\Omega)$ to itself}
\label{ap.Lp}

The goal here is to prove the boundedness of $D(\lambda-\Delta)^{-1}D^*$ from $L^p(\Omega)$ to itself, and %Lemma~\ref{lem:bdp}, i.e.,
%\[
%\|D(\lambda-\Delta)^{-1}D^*F\|_p\leq C_p \|F\|_p
%\]
%for $C_p>0$ uniformly in $\lambda\in [0,1], F\in L^p(\Omega)$.
 we will borrow a deterministic estimate from \cite{biskup2014central}.
 
 For $L\in \Z_+$, let $\Gamma_L$ be a square box
 \[
\Lambda_L=[-L,L)^d\cap Z^d.
 \]
For $\tilde{F}\in L^1(\Lambda_L, \R^d)$, we define the following integral operator for $\lambda>0$:
 \begin{equation}
 \K_{\lambda,L} \tilde{F}=\nabla (\lambda-\Delta)^{-1}\nabla^* \tilde{F}=\sum_{y\in \Lambda_L} \nabla\nabla G_\lambda(\cdot,y)\tilde{F}(y),
 \label{eq:defK}
 \end{equation}
 where $G_\lambda$ is the Green function of $\lambda-\Delta$ on $\Z^d$. We claim that the following weak type-$(1,1)$ estimate holds:
 \begin{equation}
 |\{x\in \Lambda_L: |\K_{\lambda,L} \tilde{F}(x)|>\alpha\}|\leq \frac{C}{\alpha}\sum_{x\in \Lambda_L}|\tilde{F}(x)|
 \label{eq:wk11}
 \end{equation}
 for some $C>0$ independent of $\alpha,\lambda,L>0$. In \cite[Lemma 4.6]{biskup2014central}, \eqref{eq:wk11} was shown with $G_\lambda$ in \eqref{eq:defK} replaced by the Green function of Laplacian with zero boundary condition in $\Lambda_L$. The same proof works in our case since the only ingredient we need to change in their proof is the following bound on the triple gradient of $G_\lambda$, which was given by \cite[Lemma 4.9]{biskup2014central}:
 \begin{equation}
 |\nabla_{y,i}\nabla_{x,j}\nabla_{y,k} G_\lambda(x,y)|\leq C|x-y|^{-d-1}
 \label{eq:3gr}
 \end{equation}
 for some $C>0$ independent of $\lambda>0,i,j,k=1,\ldots,d$.
 
 The estimate \eqref{eq:wk11} on the physical space can be lifted up to the probability space:
 \begin{lem}
 For any $F\in L^1(\Omega,\R^d)$, we have
 \[
 \P(|D(\lambda-\Delta)^{-1}D^*F|>\alpha)\leq \frac{C}{\alpha}\la|F|\ra
 \]
 for some constant $C>0$ independent of $\lambda,\alpha>0$.
 \label{lem:wkbd}
 \end{lem}

\begin{proof}
We fix $\lambda,\alpha>0$. First, let $\tilde{F}(x)=F(\tau_x\omega)$ for $x\in \Z^d$. Since $F \in L^1(\Omega,\R^d)$, for almost every $\omega\in \Omega$, we have $\tilde{F}\in L^1(\Lambda_L,\R^d)$ for any $L\in \Z_+$, so by \eqref{eq:wk11}
\[
\sum_{x\in \Lambda_L}1_{|\K_{\lambda,L}\tilde{F}(x)|>\alpha}\leq \frac{C}{\alpha}\sum_{x\in \Lambda_L} |\tilde{F}(x)|.
\]
Taking expectation on both sides, we derive
\[
\frac{1}{|\Lambda_L|}\sum_{x\in \Lambda_L}\P(|\K_{\lambda,L}\tilde{F}(x)|>\alpha)\leq \frac{C}{\alpha}\la|F|\ra.
\]
We can write
\[
\begin{aligned}
\frac{1}{|\Lambda_L|}\sum_{x\in \Lambda_L}\P(|\K_{\lambda,L}\tilde{F}(x)|>\alpha)=&\frac{1}{|\Lambda_L|}\sum_{x\in \Lambda_{L-\sqrt{L}}}\P(|\K_{\lambda,L}\tilde{F}(x)|>\alpha)\\
&+\frac{1}{|\Lambda_L|}\sum_{x\in \Lambda_L\setminus\Lambda_{L-\sqrt{L}}}\P(|\K_{\lambda,L}\tilde{F}(x)|>\alpha)\\
:=&(i)+(ii).
\end{aligned}
\]
For $(ii)$, it is clear that
\[
(ii)\leq \frac{|\Lambda_L\setminus\Lambda_{L-\sqrt{L}}|}{|\Lambda_L|}\to 0
\]
as $L\to\infty$. For $(i)$, we have
\[
\sum_{y\in\Z^d}\nabla\nabla G_\lambda(x,y)\tilde{F}(y)=\K_{\lambda,L}\tilde{F}(x)+\sum_{y\in \Z^d\setminus \Lambda_L}\nabla\nabla G_\lambda(x,y)\tilde{F}(y),
\]
then
\[
\begin{aligned}
\P(|\sum_{y\in\Z^d}\nabla\nabla G_\lambda(x,y)\tilde{F}(y)|>2\alpha)&\leq \P(|\K_{\lambda,L}\tilde{F}(x)|>\alpha)\\
& +\P(|\sum_{y\in \Z^d\setminus \Lambda_L}\nabla\nabla G_\lambda(x,y)\tilde{F}(y)|>\alpha).
\end{aligned}
\]
By stationarity, 
\[
\P(|\sum_{y\in\Z^d}\nabla\nabla G_\lambda(x,y)\tilde{F}(y)|>2\alpha)=\P(|D(\lambda-\Delta)^{-1}D^*F|>2\alpha)
\]
is independent of $x\in \Z^d$. For the summation outside $\Lambda_L$, we have
\[
\begin{aligned}
\P(|\sum_{y\in \Z^d\setminus \Lambda_L}\nabla\nabla G_\lambda(x,y)\tilde{F}(y)|>\alpha)\leq &\frac{1}{\alpha}\la\sum_{y\in \Z^d\setminus \Lambda_L}|\nabla\nabla G_\lambda(x,y)\tilde{F}(y)|\ra\\
\leq &\frac{1}{\alpha}\sum_{y\in \Z^d\setminus \Lambda_L}|\nabla\nabla G_\lambda(x,y)|\to 0
\end{aligned}
\]
as $L\to\infty$, uniformly in $x\in  \Lambda_{L-\sqrt{L}}$. 

Now we have
\[
\begin{aligned}
&\frac{|\Lambda_{L-\sqrt{L}}|}{|\Lambda_L|}\P(|D(\lambda-\Delta)^{-1}D^*F|>2\alpha)-\frac{1}{\alpha|\Lambda_L|}\sum_{x\in \Lambda_{L-\sqrt{L}}}\sum_{y\in \Z^d\setminus \Lambda_L}|\nabla\nabla G_\lambda(x,y)|\\
\leq &\frac{C}{\alpha}\la|F|\ra-(ii).
\end{aligned}
\]
By sending $L\to\infty$, we obtain
\[
\P(|D(\lambda-\Delta)^{-1}D^*F|>2\alpha)\leq \frac{C}{\alpha}\la |F|\ra,
\]
which completes the proof.
\end{proof}

Using Lemma~\ref{lem:wkbd} and the fact that $D(\lambda-\Delta)^{-1}D^*$ is bounded from $L^2(\Omega,\R^d)$ to itself, we can apply the standard interpolation argument, e.g., \cite[Theorem 4.4]{biskup2014central}, to conclude that $D(\lambda-\Delta)^{-1}D^*$ is bounded from $L^p(\Omega,\R^d)$ to itself for any $p\in (1,\infty)$. 

For the case $\lambda=0$, we only need to note that for any $F\in L^2(\Omega,\R^d)$,
\[
D(\lambda-\Delta)^{-1}D^*F\to D(-\Delta)^{-1}D^*F
\]
in $L^2(\Omega,\R^d)$ as $\lambda\to 0$, so by applying Fatou's Lemma, we conclude $D(-\Delta)^{-1}D^*$ is bounded from $L^p(\Omega,\R^d)$ to itself. The proof of Lemma~\ref{lem:bdp} is complete.

%Let $T_N$ be the $d-$dimensional discrete torus:
%\[
%T_N=\{-N,\ldots,N-1\}^d,
%\]
%and $U:T_N\to\R$ solve
%\begin{equation}
%(\lambda-\tilde{\Delta})U_\lambda=\tilde{\nabla}^*\cdot h
%\label{eq:dispo}
%\end{equation}
%for some $h:T_N\to\R$. We denote $\tilde{\nabla},\tilde{\nabla}^*$ are the discrete gradient and divergence on $T_N$ and $-\tilde{\Delta}=\tilde{\nabla}^*\tilde{\nabla}$.
%
% Without loss of generality we can assume $h$ has mean zero, i.e.,
%\[
%\sum_{x\in T_N}h(x)=0.
%\]
%When $\lambda=0$, we further assume $U_\lambda$ has mean zero to guarantee the uniqueness of the solution to \eqref{eq:dispo}. 
%
%One of the main technical estimates in \cite{} is the following weak type-$(1,1)$ bound: 
%\begin{equation}
%\frac{1}{|T_N|}\sum_{x\in T_N}1_{ |D U_\lambda|\geq t}\leq \frac{C}{t}\frac{1}{|T_N|}\sum_{x\in T_N} |h(x)|
%\label{eq:wk11}
%\end{equation}
%with $C$ independent of $\lambda\in [0,1]$ and $t>0$. For our equation on $\Omega$, we have a similar result:
%\begin{lem}
%Let $\phi_\lambda\in L^2(\Omega)$ solving 
%\[
%(\lambda-\Delta)\phi_\lambda=D^*F
%\]
%with $F\in L^2(\Omega)$, then for any $t>0$, we have
%\[
%\P(|D\phi_\lambda|\geq t)\leq \frac{C}{t}\E\{|F|\}
%\]
%\end{lem}
%
%\begin{proof}
%Since $F\in L^2(\Omega)$ and we have the i.i.d. structure, by density argument we can assume $F$ only depends on finitely many $\zeta_e$. For any $N$, we can define the periodized environment by $\pi_N \zeta$ by 
%\[
%(\pi_N\zeta)_e=
%\]
%
%\end{proof}

\section*{Acknowledgments}  
We would like to thank the Mathematics Research Center of Stanford University for the hospitality during JCM's visit. Y.G. was partially supported by the NSF through DMS-1613301. We would like to thank the anonymous referee for his careful reading of the paper and helpful suggestions.

%\bibliographystyle{abbrv}
%\bibliography{Markov}

\begin{thebibliography}{99}

\bibitem{berbis} N.\ Berger, M.\ Biskup. Quenched invariance principle for simple random walk on percolation clusters. \emph{Probab.\ Theory Related Fields} \textbf{137} (1-2), 83-120 (2007). 

\bibitem{biskup2014central}
M.~Biskup, M.~Salvi, and T.~Wolff.
\newblock A central limit theorem for the effective conductance: linear
  boundary data and small ellipticity contrasts.
\newblock {\em Comm. Math. Phys.}, 328(2):701--731, 2014.

\bibitem{courrege}
P.~Courr\`ege.
\newblock Sur la forme int\'egro-diff\'erentielle des op\'erateurs de
  $C^\infty_k$ dans $C$ satisfaisant au principe du maximum.
\newblock In {\em S\'eminaire de {T}h\'eorie du {P}otentiel, {D}irig\'e par
  {M}. {B}relot, {G}. {C}hoquet et {J}. {D}eny, 1965/66}, tome 10, expos\'e 2, 
  1--38. Secr\'etariat math\'ematique, Paris, 1966.


\bibitem{dyn}
E.~B. Dynkin.
\newblock Markov processes and random fields.
\newblock {\em Bull. Amer. Math. Soc. (N.S.)}, 3(3):975--999, 1980.

\bibitem{fukushima}
M.~Fukushima, Y.~{\=O}shima, and M.~Takeda.
\newblock {\em Dirichlet forms and symmetric {M}arkov processes}, volume~19 of
  {\em de Gruyter Studies in Mathematics}.
\newblock Walter de Gruyter \& Co., Berlin, 1994.


\bibitem{gos} G.\ Giacomin, S.\ Olla, H.\ Spohn. Equilibrium fluctuations for $\nabla \varphi$ interface model. \emph{Ann.\ Probab.}\ \textbf{29} (3), 1138--1172 (2001). 

\bibitem{gloria2011optimal}
A.~Gloria and F.~Otto.
\newblock An optimal variance estimate in stochastic homogenization of discrete
  elliptic equations.
\newblock {\em Ann. Probab.}, 39(3):779--856, 2011.

\bibitem{gu-mourrat}
Y.~Gu and J.-C.\ Mourrat. Scaling limit of fluctuations in stochastic homogenization.
\newblock {\em Multiscale Model. Simul.}, 14(1):452--481, 2016.

\bibitem{kalman}
G.~Kallianpur and V.~Mandrekar.
\newblock The {M}arkov property for generalized {G}aussian random fields.
\newblock {\em Ann. Inst. Fourier (Grenoble)}, 24(2):vi, 143--167, 1974.
\newblock Colloque International sur les Processus Gaussiens et les
  Distributions Al{\'e}atoires (Colloque Internat. du CNRS, No. 222,
  Strasbourg, 1973).

\bibitem{kolfried}
D.~Koller and N.~Friedman.
\newblock {\em Probabilistic graphical models}.
\newblock Adaptive Computation and Machine Learning. MIT Press, Cambridge, MA,
  2009.


\bibitem{kol}
T.~Kolsrud.
\newblock On the {M}arkov property for certain {G}aussian random fields.
\newblock {\em Probab. Theory Related Fields}, 74(3):393--402, 1987.

\bibitem{kuen} R.\ K\"unnemann. The diffusion limit for reversible jump processes on $\Z^d$ with ergodic random bond conductivities. \emph{Comm. Math. Phys.}, 90(1):27--68, 1983. 

\bibitem{kuensch}
H.~K{\"u}nsch.
\newblock Gaussian {M}arkov random fields.
\newblock {\em J. Fac. Sci. Univ. Tokyo Sect. IA Math.}, 26(1):53--73, 1979.

\bibitem{levy}
P.~L{\'e}vy.
\newblock {\em Processus {S}tochastiques et {M}ouvement {B}rownien. {S}uivi
  d'une note de {M}. {L}o\`eve}.
\newblock Gauthier-Villars, Paris, 1948.

\bibitem{mck}
H.~P. McKean, Jr.
\newblock Brownian motion with a several-dimensional time.
\newblock {\em Teor. Verojatnost. i Primenen.}, 8:357--378, 1963.
\newblock Engl.\ transl. {\em Theor. Probability Appl.} 8:335--354, 1963.

\bibitem{mil}
J.~Miller.
\newblock Fluctuations for the {G}inzburg-{L}andau {$\nabla\phi$} interface
  model on a bounded domain.
\newblock {\em Comm. Math. Phys.}, 308(3):591--639, 2011.

\bibitem{mol}
G.~M. Mol{\v{c}}an.
\newblock Some problems connected with the {B}rownian motion of {L}\'evy.
\newblock {\em Teor. Verojatnost. i Primenen.}, 12:747--755, 1967.
\newblock Engl.\ transl. {\em Theor. Probability Appl.} 12:682--690, 1967.

\bibitem{MN}
J.-C. Mourrat and J.~Nolen. Scaling limit of the corrector in stochastic homogenization. Preprint, arXiv:1502.07440 (2015).

\bibitem{mourrat2014correlation}
J.-C. Mourrat and F.~Otto. Correlation structure of the corrector in
  stochastic homogenization. 
  \newblock {\em Ann. Probab.}, 44(5):3207--3233, 2016.

\bibitem{nadspe} A.\ Naddaf, T.\ Spencer.
On homogenization and scaling limit of some gradient perturbations of a massless free field. \emph{Comm.\ Math.\ Phys.}\ \textbf{183} (1), 55--84 (1997). 

\bibitem{papanicolaou1979boundary}
G.~C. Papanicolaou and S.~R.~S. Varadhan.
\newblock Boundary value problems with rapidly oscillating random coefficients.
\newblock In {\em Random fields, {V}ol. {I}, {II} ({E}sztergom, 1979)},
  volume~27 of {\em Colloq. Math. Soc. J\'anos Bolyai}, pages 835--873.
  North-Holland, Amsterdam, 1981.

\bibitem{pit1}
L.~D. Pitt.
\newblock A {M}arkov property for {G}aussian processes with a multidimensional
  parameter.
\newblock {\em Arch. Rational Mech. Anal.}, 43:367--391, 1971.

\bibitem{kotpre}
D.~Preiss and R.~Koteck\'y.
\newblock Markoff property of generalized random fields.
\newblock In {\em Seventh {W}inter {S}chool on {A}bstract {A}nalysis}, pages
  61--66. Czechoslovak Academy of Sciences, Mathematical Institute, Prague,
  1979.

\bibitem{roc}
M.~R{\"o}ckner.
\newblock Generalized {M}arkov fields and {D}irichlet forms.
\newblock {\em Acta Appl. Math.}, 3(3):285--311, 1985.

\bibitem{sheffield} S.\ Sheffield. 
\newblock Gaussian free fields for mathematicians. 
\newblock\emph{Probab.\ Theory Related Fields} \textbf{139} (3-4), 521-541 (2007). 

\end{thebibliography}

\end{document}